\documentclass[12pt, leqno]{article}
\usepackage{amsmath}
\usepackage{amsthm}
\usepackage{amssymb}
\usepackage{latexsym}
\usepackage{graphicx}
\allowdisplaybreaks[1]
\usepackage{amsfonts}
\usepackage{ascmac}
\usepackage{color}
\usepackage{enumerate}

\theoremstyle{definition}
\theoremstyle{plain}
\allowdisplaybreaks[1]
\date{}
\newtheorem{Thm}{Theorem}[section]

\newtheorem{Lemma}[Thm]{Lemma}
\newtheorem{Cor}[Thm]{Corollary}
\newtheorem{Def}[Thm]{Definition}

\newcommand{\dis}{\displaystyle}

\newcommand{\norm}{\parallel}

\newcommand{\Z}{{\mathbb Z}}

\newcommand{\T}{{\mathbb T}}
\newcommand{\N}{{\mathbb N}}
\newcommand{\R}{{\mathbb R}}

\newcommand{\D}{ \mathcal{D}}

\newcommand{\ep}{\varepsilon }
\newcommand{\2}{\frac{1}{2} }
\newcommand{\wto}{\rightharpoonup}
\newcommand{\tu}{\tilde{u}}
\newcommand{\Omegain}{{\Omega_h\setminus\partial\Omega_h}}

\newcommand{\nnorm}{|\!|\!|}

\def\text#1{\mbox{#1 }}

\title{\bf More on convergence of \\
Chorin's projection method for incompressible Navier-Stokes equations
}
\author{Masataka Maeda 
\footnote{Department of Mathematics, Faculty of Science and Technology, Keio University, 3-14-1 Hiyoshi, Kohoku-ku, Yokohama, 223-8522, Japan. E-mail:  masataka-918.m@keio.jp  (This work was done when he belonged to Keio University.)}\,\,\,
and Kohei Soga
\footnote{Department of Mathematics, Faculty of Science and Technology, Keio University, 3-14-1 Hiyoshi, Kohoku-ku, Yokohama, 223-8522, Japan. E-mail:  
 soga@math.keio.ac.jp  }}
\begin{document}
\maketitle
\begin{abstract} 
\noindent Kuroki and Soga [Numer. Math. 2020] proved that a version of Chorin's fully discrete projection method, originally introduced by  A. J. Chorin [Math. Comp. 1969],  is unconditionally  solvable  and convergent  within an arbitrary fixed time interval to a Leray-Hopf weak solution of the incompressible Navier-Stokes equations on a bounded domain with an arbitrary external force. This paper is a continuation of Kuroki-Soga's work. We show time-global solvability and convergence of our scheme; $L^2$-error estimates for the scheme in the class of  smooth exact solutions; application of the scheme to the problem with a time-periodic external force to investigate time-periodic (Leray-Hopf weak) solutions, long-time behaviors, error estimates, etc.

\medskip

\noindent{\bf Keywords:} fully discrete projection method; incompressible Navier-Stokes equations; Leray-Hopf weak solution; time-periodic solution; error estimate  \medskip

\noindent{\bf AMS subject classifications:}  35Q30; 35D30; 65M06; 65M15
\end{abstract}
%
\setcounter{section}{0}
\setcounter{equation}{0}
\section{Introduction}

We consider the incompressible Navier-Stokes equations on a bounded domain of $\R^3$
\begin{eqnarray}\label{NS}
 \left \{
\begin{array}{lll}
\,\,\,\,\, v_t&=& -(v\cdot \nabla)v +\Delta v+f -\nabla p\mbox{\quad in $(0,T]\times\Omega$  or $(0,\infty)\times\Omega$,}
\medskip\\
\nabla\cdot v &=&0\mbox{\quad\quad\quad\quad\quad\quad\quad\quad\quad\quad\quad\,\,\, in $(0,T]\times\Omega$ or $(0,\infty)\times\Omega$,}
\medskip\\
v(0,\cdot)&=&v^0\mbox{\qquad\qquad\qquad\qquad\qquad\,\,\,\,\,\,\, in $\Omega$},
\medskip\\
\,\,\,\,\,v&=&0\mbox{\qquad\qquad\qquad\qquad\qquad\,\,\,\,\,\,\,\,\, on $\partial \Omega$},
\end{array}
\right.\\\nonumber  
\,\,\,\,\,\,\,\, \Omega\subset \R^3 \mbox{ is a bounded connected open set with a Lipschitz boundary,}
\end{eqnarray}
where $v=v(t,x)$ is the velocity, $p=p(t,x)$ is the pressure, $f=f(t,x)$ is a given external force, $T$ is an arbitrary positive number, $v^0$ is initial data and $v_t=\partial_t v$, $v_{x_j}=\partial_{x_j}v$, etc., stand for  the partial (weak) derivatives of $v(t,x)$. 
Let $f$ and $v^0$ be arbitrarily taken as  
$$\mbox{$f\in L^2_{\rm loc}([0,\infty);L^2(\Omega)^3)$,\quad $v^0\in L^2_{\sigma}(\Omega)$}.$$ 
Here, $C^r_{0}(\Omega)$  is the family of $C^r$-functions\,:\,$\Omega\to\R$ with a compact support; $C^r_{0,\sigma}(\Omega):=\{v\in C^r_0(\Omega)^3\,|\,\nabla\cdot v=0\}$;  $H^1_0(\Omega)$ is the closure of $C^\infty_0(\Omega)$ with respect to the norm  $\norm \cdot\norm_{H^1(\Omega)}$; $L^2_{\sigma}(\Omega)$ (resp. $H^1_{0,\sigma}(\Omega)$) is the closure of $C^\infty_{0,\sigma}(\Omega)$ with respect to the norm $\norm\cdot\norm_{L^2(\Omega)^3}$ (resp. $\norm \cdot\norm_{H^1(\Omega)^3}$). 

A function $v=(v_1,v_2,v_3):[0,T]\times\Omega\to\R^3$ is called a {\it Leray-Hopf weak solution} of \eqref{NS}, if 
\begin{eqnarray}\nonumber
&&v\in L^\infty([0,T];L^2_{\sigma}(\Omega))\cap L^2([0,T];H^1_{0,\sigma}(\Omega)),\\\label{weak-form-NS}
&&-\int_\Omega v^0(x)\cdot\phi(0,x)dx- \int_0^T\int_\Omega v(t,x)\cdot \partial_t\phi(x,t)dxdt \\\nonumber
&&=-\sum_{j=1}^3\int_0^T\int_\Omega v_j(t,x)\partial_{x_j}v(t,x)\cdot\phi(t,x)dxdt\\\nonumber
&&\quad -\sum_{j=1}^3\int_0^T\int_\Omega\partial_{x_j}v(t,x)\cdot\partial_{x_j}\phi(t,x)dxdt\\\nonumber
&&\quad +\int_0^T\int_\Omega f(t,x)\cdot\phi(t,x)dxdt\quad \mbox{ for all $\phi\in C^\infty_0((-1,T);C^\infty_{0,\sigma}(\Omega))$,}
\end{eqnarray} 
where $x\cdot y:=\sum_{i=1}^3x_iy_i$ for $x,y\in\R^3$. 

A function $v$ belonging to  $L^\infty([0,\infty);L^2_\sigma(\Omega))\cap L^2_{\rm loc}([0,\infty);H^1_{0,\sigma}(\Omega))$ is  called a {\it time global Leray-Hopf weak solution} of \eqref{NS}, if  $v|_{[0,T]}$ satisfies (\ref{weak-form-NS})  for each fixed $T>0$.  

This paper is a continuation of the work \cite{Kuroki-Soga}. In \cite{Kuroki-Soga}, Kuroki-Soga proposed a version of Chorin's fully discrete projection method applied to \eqref{NS} and proved its convergence within an arbitrarily fixed time interval   to a Leray-Hopf weak solution (up to a subsequence) by means of a new compactness argument (the standard Aubin-Lions-Simon approach fails).    
It seems that Chorin's fully discrete projection method is no longer very popular in modern computational fluid dynamics because of its less accuracy, i.e., discretization of $\Omega$ into a uniform mesh and the Dirichlet boundary condition cause a less accurate result. However, we believe that {\it Chorin's fully discrete projection method can be one of strong mathematical tools to analyze the Navier-Stokes equations including complicated issues such as free boundary problems, long time behaviors, time-periodic solutions, bifurcations, etc.} Unlike Galerkin type methods, the projection method solves the equations more directly, which could be an advantage for better understandings. Motivated by such an opinion, we further develop mathematical analysis of Chorin's fully discrete projection method beyond the convergence to a Leray-Hopf weak solution of the initial boundary value problem. 

In Section 2, we first formulate  a version of Chorin's fully discrete projection method and recall the results in \cite{Kuroki-Soga}. Note that \cite{Kuroki-Soga} deals with the one-sided difference and the discrete Helmholtz-Hodge decomposition formulated by the zero Dirichlet boundary condition for both of the divergence-free part and  potential part. Here, we deal with the central difference and the discrete Helmholtz-Hodge decomposition formulated by the zero Dirichlet boundary condition for   the divergence-free part and the zero mean condition for the potential part. This modification in the discrete Helmholtz-Hodge decomposition is particularly important to obtain error estimates, since the exact pressure term $p$ does not necessarily satisfy  the zero Dirichlet boundary condition. The new result of Section 2 is the time-global solvability of our discrete problem with a fixed discretization parameter, under the assumption that the $L^2$-norm of  the external force within $[t,t+1]\times\Omega$ is uniformly bounded for any $t\ge0$. This result yields a sequence of step functions that is  convergent locally in time to a time-global Leray-Hopf weak solution.           

In Section 3, we demonstrate an error estimate for our scheme in the $C^3$-class. In \cite{Chorin}, Chorin showed an $L^2$-error estimate of $O(h^2)$ in the $C^5$-class for problems with the periodic boundary conditions, where $h>0$ and $\tau>0$ are the mesh size for the space variables and time variable, respectively. In the case of the zero Dirichlet boundary condition, the issue is more complicated due to the gap between the exact boundary $\partial\Omega$ and the boundary of the grid space. Semi-discrete projection methods, i.e., discrete in time with the mesh size $\tau>0$ and continuous in space, are free from  this complication and one can do a lot also in the class of strong solutions. In fact, Rannacher \cite{Rannacher} gave an error estimate of $O(\tau)$ for the Dirichlet problem. Since Chorin took the diffusive scaling condition $\tau=O(h^2)$ in his fully discrete setting, the two results by Chorin and  Rannacher seem to be ``consistent''. We also refer to Shen \cite{Shen} and the references therein for further investigation on semi-discrete projection methods. Although a fully discrete projection method applied to the  Dirichlet problem is said to be less accurate,  to the best of the authors' knowledge, there is no rigorous error analysis. We will show an $L^2$-error  estimate of $O(h^{\frac{1}{4}})$ for a discrete solution and exact $C^3$-solution under the scaling condition $\tau=O(h^{\frac{3}{4}})$. Note that Chorin \cite{Chorin} and Temam \cite{Temam-2} proved convergence of their schemes  with the standard diffusive scaling condition, while Kuroki-Soga \cite{Kuroki-Soga} gave scale-free results; The diffusive scaling does not yield such an error estimate in our formulation.  
We will see that our error bound and scaling condition arise from the discrete Helmholtz-Hodge decomposition, not from the discrete Navier-Stokes equations. 
Although the error estimates of $O(h^\frac{1}{4})$ does not sound very sharp, the proof provides a new idea to estimate a remainder term on the boundary arising from ``summation by parts'' in the discrete problem, which is reminiscent of the construction of the trace operator. This idea would provide further applications in analysis of finite difference methods.     

In Section 4, we apply our scheme to investigation of the problem with a time-periodic external force. In this problem, one of the main issues is to find a time-periodic solution with the same time-period as that of the external force. After the first attempt by Serrin \cite{Serrin}, many results have been obtained. We refer to Kyed \cite{Kyed} for a nice review of the literature on time-periodic solutions to the Navier-Stokes equations. To the best of the authors' knowledge, there is no mathematical investigation of time-periodic solutions in terms of the fully discrete projection method. We find a discrete time-periodic solution as a fixed point of the time-$1$ map of the discrete Navier-Stokes equations and prove convergence to a time-periodic Leray-Hopf weak solution. We also investigate long-time behaviors of  discrete solutions, assuming that there exists a ``small'' discrete solution in the $L^\infty$-sense. We obtain exponential contraction of any other discrete solutions.  Since the rate of contraction is independent of the size of the discretization parameters, we see that similar exponential contraction holds for exact Leray-Hopf weak solutions, where we do not assume any regularity except for the $L^\infty$-bound of a solution. This idea would provide further applications in analysis of stability of a time-periodic solution, its bifurcation, etc. of the exact problem through the discrete problem. 
Furthermore, we prove that any discrete solution falls into the $O(h^{\frac{1}{4}})$-neighborhood of an exact time-periodic solution, provided the exact solution is of the $C^3$-class and ``small''. These results can be seen as a version of the results by Serrin \cite{Serrin}, Miyakawa-Teramoto \cite{MT} and Teramoto \cite{Teramoto}. We refer also to Jauslin-Kreiss-Moser \cite{JKM} and Nishida-Soga \cite{Nishida-Soga} for similar investigations on time-periodic entropy solutions of forced Burgers equations through finite difference methods. Finally, we point out Kagei-Nishida-Teramoto \cite{Kagei-Nishida-Teramoto} for analysis on stability of stationary solutions of the incompressible Navier-Stokes equations via the corresponding artificial compressible system.       

In Section 5, we briefly state results corresponding to Section 3 and 4 in the case of the periodic boundary conditions, where the diffusive scaling  and the central difference play an essential role. Since the $L^2$-estimates can be sharpened to be $O(h^2)$, we obtain an $L^\infty$-estimate of $O(\sqrt{h})$ through the inequality used by Chorin \cite{Chorin}.       

\setcounter{section}{1}
\setcounter{equation}{0}
\section{Construction of Leray-Hopf weak solution}
We investigate a version of the scheme studied in \cite{Kuroki-Soga} with the central difference, as well as the discrete Helmholtz-Hodge decomposition with the zero Dirichlet boundary condition for the divergence-free part and the zero mean condition for the potential part. 
\subsection{Calculus on grid}
Let $h>0$ be the mesh size for the space variables and consider the grid 
$$h\Z^3:=\{ (hz_1,hz_2,hz_3)\,|\,z_1,z_2,z_3\in\Z \}.$$
Let $e^1,e^2,e^3$ be the standard basis of $\R^3$. For $B\subset h\Z^3$, the boundary $\partial B$ of $B$ is defined as $\partial B:=\{x \in B\,|\,\{x \pm he^i \}_{i=1,2,3}\not\subset B \}$.
Let $\Omega$ be a bounded connected open subset of $\R^3$ with a Lipschitz boundary $\partial\Omega$. For $x\in\R^3$ and $r>0$, set 
\begin{eqnarray*}
C_r(x)&:=&\Big[x_1-\frac{r}{2},x_1+\frac{r}{2}\Big]\times\Big[x_2-\frac{r}{2},x_2+\frac{r}{2}\Big]\times\Big[x_3-\frac{r}{2},x_3+\frac{r}{2}\Big],\\
\Omega_h&:=&\{x\in\Omega\cap h\Z^3\,|\,C_{4h}(x)\subset\Omega\}.
\end{eqnarray*}
Define the discrete derivatives of a function $\phi:\Omega_h\to\R$ as follows:  For each $x\in \Omega_h$,  
\begin{eqnarray*}
&&D_i^+\phi(x):=\frac{\phi(x+he^i)-\phi(x)}{h},\,\,\,D_i^-\phi(x):=\frac{\phi(x)-\phi(x-he^i)}{h},\\
&&D_i\phi(x):=\frac{\phi(x+he^i)-\phi(x-he^i)}{2h},\\
&&D_i^2\phi(x):=\frac{\phi(x+he^i)+\phi(x-he^i)-2\phi(x)}{h^2}
\end{eqnarray*}  
where  {\it these operations work under the condition that $\phi$ is extended to be $0$ outside $\Omega_h$, i.e., $\phi(x+ he^i)=0$ (resp. $\phi(x- he^i)=0$) if $x+ he^i\not\in \Omega_h$ (resp. $x- he^i\not\in \Omega_h$).} For $x,y\in\R^d$, set $x\cdot y:=\sum_{i=1}^dx_iy_i$, $|x|:=\sqrt{x\cdot x}$. Define the discrete gradient of a function $\phi:\Omega_h\to\R$ and the discrete divergence of a function $w=(w_1,w_2,w_3):\Omega_h\to\R^3$ as 
$$\D\phi(x):=(D_1\phi(x),D_2\phi(x),D_3\phi(x)),\,\,\,\D\cdot w(x):=D_1w_1(x)+D_2w_2(x)+D_3w_3(x).$$
We often use discrete versions of integration by parts, ``summation by parts'', where we need careful treatments of reminder terms on the boundary. For this purpose, we introduce the following notation (see with Figure 1):   
\begin{eqnarray*}
\Gamma_h^{i+}&:=&\{x\in\partial\Omega_h\,|\,x-he^i\in\Omega_h\setminus\partial\Omega_h\},\\
\tilde{\Gamma}_h^{i+}&:=& \{x\in\partial\Omega_h\,|\,x+he^i\not\in\Omega_h\}    ,\\
\Gamma_h^{i-}&:=&  \{x\in\partial\Omega_h\,|\,x+he^i\in\Omega_h\setminus\partial\Omega_h\}   ,\\
\tilde{\Gamma}_h^{i-}&:=&  \{x\in\partial\Omega_h\,|\,x-he^i\not\in\Omega_h\},
\end{eqnarray*}
e.g.,  if there is a sequence of points of $\Omega_h$ on a line parallel to $e^i$ as shown in Figure 1, we have  $x^+\in\Gamma_h^{i+}$, $x^+\not\in\tilde{\Gamma}_h^{i+}$,  $\tilde{x}^+\not\in\Gamma_h^{i+}$,  $\tilde{x}^+\in\tilde{\Gamma}_h^{i+}$, 
$x^-\in\Gamma_h^{i-}$, $x^-\not\in\tilde{\Gamma}_h^{i-}$,  $\tilde{x}^-\not\in\Gamma_h^{i-}$,  $\tilde{x}^-\in\tilde{\Gamma}_h^{i-}$. 
\begin{center}
\includegraphics[bb=0 0 820 151, width=10cm]{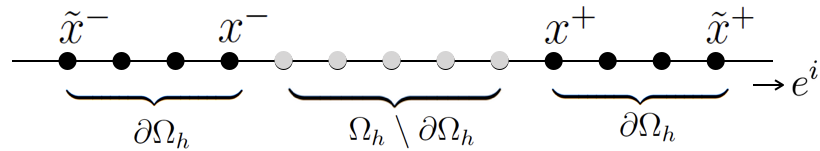}
\center{Figure 1.}
 \end{center}

\begin{Lemma}\label{vector-cal1}
For functions $\phi:\Omega_h\to\R$ and $u:\Omega_h\to\R^3$, we have 
$$ \sum_{x\in \Omega_h}u(x) \cdot \D\phi(x)=-\sum_{x\in \Omega_h}(\D\cdot u(x) ) \phi(x).$$
\end{Lemma}
\begin{proof}
We may carry out $\sum_{x\in \Omega_h}$ by the summation along sequences of grid points of $\Omega_h$ on each line parallel to $e^i$ ($i=1,2$ or $3$). By shifting $x$ to $x\pm he^i$ in the summation, we obtain    
\begin{eqnarray*}
&&\sum_{x\in \Omega_h} u(x)\cdot \D\phi(x)
=\sum_{i=1}^3\Big(\sum_{x\in \Omega_h} u_i(x)\phi(x+he^i)-\sum_{x\in\Omega_h}u_i(x)\phi(x-he^i)\Big)\frac{1}{2h}\\
&&= -\sum_{x\in \Omega_h} \D\cdot u(x)\phi(x)+ 
\sum_{i=1}^3\Big(\sum_{x\in \tilde{\Gamma}_h^{i+}}u_i(x)\phi(x+he^i)\\ \nonumber
&&\quad -\sum_{x\in \tilde{\Gamma}_h^{i-}}u_i(x-he^i)\phi(x)
-\sum_{x\in \tilde{\Gamma}_h^{i-}}u_i(x)\phi(x-he^i)
+\sum_{x\in \tilde{\Gamma}_h^{i+}}u_i(x+he^i)\phi(x)
\Big)\frac{1}{2h}.
\end{eqnarray*}
Since $\phi$ and $u$ are supposed to be $0$ outside $\Omega_h$, we obtain the assertion. 
\end{proof}
We give two Poincar\'e type inequalities for functions on the grid. 
\begin{Lemma}[{\bf Poincar\'e type inequality I}]\label{Poincare}
For each function $\phi:\Omega_h\to\R$ with $\phi|_{\partial \Omega_h}=0$, we have 
\begin{eqnarray*}
\sum_{x\in\Omega_h} |\phi(x)|^2\le A^2 \sum_{x\in\Omega_h\setminus\partial\Omega_h}  |D_i^+ \phi(x)|^2,\quad i=1,2,3,
\end{eqnarray*}
where $A>0$ is a constant depending only on the diameter of $\Omega$. 
\end{Lemma}
\begin{proof}
This is proved in \cite{Kuroki-Soga}.
\end{proof}
The second one is for functions without the zero boundary condition, where the mean value of each function is involved in the inequality. This is essentially applied to the discrete Helmholtz-Hodge decomposition below, where we need to avoid the presence of values of the derivatives on $\partial\Omega_h$ (see Theorem \ref{22estimate}). Since we formulate the inequality with the central difference, we must look at the $\{2e^1,2e^2,2e^3\}$-translation invariance in the grid $h\Z^3$: $h\Z^3$ is divided into $G^1,\ldots,G^8$, each of which is invariant under the $\{2e^1,2e^2,2e^3\}$-translation, i.e.,   $G^1,\ldots,G^8$ are the sets of grid points with index $($even$,$ even$,$ even$)$, $($even$,$ even$,$ odd$)$, $($even$,$ odd$,$ even$)$, $($odd$,$ even$,$ even$)$, $($even$,$ odd$,$ odd$)$, $($odd$,$ odd$,$ even$)$, $($odd$,$ even$,$ odd$)$, $($odd$,$ odd$,$ odd$)$, respectively. Introduce the following notation: 
\begin{eqnarray*}
\Omega_h^\circ&:=&\{x\in\Omegain\,|\,x+a^1he^1+a^2he^2+a^3he^3\in\Omegain, \,\,\,a^1,a^2,a^3=0,1,2     \}, \\
\Omega_h^{\circ j}&:=&\Omega_h^\circ\cap G^j\quad(j=1,\cdots,8).
\end{eqnarray*}
The reason why we introduce $\Omega_h^\circ$ is to avoid the presence of values of the derivatives on $\partial\Omega_h$ in the Poincar\'e type inequality. 
For each $x\in \partial\Omega_h$, there exists $\omega\in\{\pm e^j\}_{j=1,2,3}$ such that $x+h\omega\not\in\Omega_h$; Then, $C_{4h}(x+h\omega)\not\subset\Omega$ and there exists $x^\ast\in\partial\Omega\cap C_{4h}(x+h\omega)$. Hence, we have $|x^\ast-x|\le |x^\ast-(x+h\omega)|+|(x+h\omega)-x|\le (1+2\sqrt{3})h$. 
\begin{eqnarray}\label{distance1}
\mbox{For each $x\in \partial\Omega_h$, there exists $x^\ast\in\partial\Omega$ such that $|x^\ast-x|\le(1+2\sqrt{3})h$}. 
\end{eqnarray}
Similarly, for each $x\in\partial\Omega_h^{\circ}$, we have $\omega\in\{\pm e^j\}_{j=1,2,3}$ such that  $x+h\omega\not\in\Omega_h^\circ$ and $a^i\in\{0,1,2 \}$ such that $x+h\omega+\sum_{i=1}^3a^ihe^i\in\partial\Omega_h$. Hence, \eqref{distance1} implies  
\begin{eqnarray}\label{distance2}
\mbox{ For each $x\in \partial\Omega_h^{\circ}$, there exists $x^\ast\in\partial\Omega$ such that $|x^\ast-x|\le2(1+2\sqrt{3})h$}. 
\end{eqnarray}
We always assume that $h>0$ is small enough so that $\Omega_h^{\circ j}$ is connected, i.e., for any $x,\tilde{x}\in\Omega_h^{\circ j}$, we have $\omega^1,\omega^2,\ldots,\omega^K\in\{\pm e^i\}_{i=1,2,3}$ such that $x+2h\omega^1+\cdots+2h\omega^k\in\Omega_h^{\circ j}$ for all $k\le K$ and $x+2h\omega^1+\cdots+2h\omega^K=\tilde{x}$.
\begin{Lemma}[{\bf Poincar\'e type inequality II}]\label{Poincare2}
For each function $\phi:\Omega_h\to\R$, we have 
\begin{eqnarray*}
\sum_{j=1}^8\sum_{x\in\Omega_h^{\circ j}} |\phi(x)-[\phi]^j|^2
\le \tilde{A}^2 \sum_{x\in\Omega_h\setminus\partial\Omega_h}  |\D \phi(x)|^2,
\quad[\phi]^j:=(\sharp \Omega_h^{\circ j})^{-1} \sum_{x\in\Omega_h^{\circ j}}\phi(x) ,
\end{eqnarray*}
where $\tilde{A}>0$ is a constant depending only on $\Omega$. 
\end{Lemma}
\begin{proof}
See Appendix. 
\end{proof}
Next two lemmas state Lipschitz interpolation of step functions, which is used to show convergence of a step function in $H^1$.   
\begin{Lemma}\label{Lip-interpolation} 
 For a function $u:\Omega_h\to\R$ with $u|_{\partial\Omega_h}=0$ and the step function $v$ derived from $u$ as 
\begin{eqnarray*}
&&v(x):=\left\{
\begin{array}{lll}
&u(y)\mbox{\quad for $x\in {C_{h}^+(y)}$, $y\in \Omega_{h}$},
\medskip\\
&0\mbox{\quad \,\,\,\,\,\,\,\,\mbox{otherwise}}, 
\end{array}
\right. \\
&&C_{r}^+(y):=[y_1,y_1+r)\times[y_2,y_2+r)\times[y_3,y_3+r),
\end{eqnarray*}
there exists  a Lipschitz continuous function $w:\Omega \to \R$ with a compact support  such that 
\begin{eqnarray*}
&&\norm w-v\norm_{L^2(\Omega)}\le K h\sum_{j=1}^3\norm D_j^+ u\norm_{\Omega_h},\\
&&\norm \partial_{x_i}w(x)\norm_{L^2(\Omega)}\le K \sum_{j=1}^3\norm D_j^+ u\norm_{\Omega_h}, \mbox{ $i=1,2,3$},
\end{eqnarray*}
where $K$ is a constant independent of $u$ and $h$. 
\end{Lemma}
\begin{proof}
This is proved in \cite{Kuroki-Soga}.
\end{proof}
\begin{Lemma}\label{Lip-interpolation2} 
For a function $u:\Omega_h\to\R$ and  
the step function $v$ derived from $u|_{\Omega_{h}^{\circ j}}$ as 
\begin{eqnarray*}
&&v(x)=u(y)\mbox{\quad for $x\in {C_{2h}^+(y)}$, $y\in \Omega_{h}^{\circ j} $}\quad (j=1,\ldots,8),\end{eqnarray*}
there exists  a Lipschitz continuous function 
$$w:\Theta_h^j \to \R,\quad \Theta_h^j:=\bigcup_{y\in\Omega_h^{\circ j}}C^+_{2h}(y) $$
such that 
\begin{eqnarray*}
&&\norm w-v\norm_{L^2(\Theta_h^j)}^2\le \tilde{K} h \norm \D u\norm_{\Omega_h\setminus\partial\Omega_h},\\
&&\norm \partial_{x_i}w(x)\norm_{L^2(\Theta_h^j)}\le \tilde{K}\norm \D u\norm_{\Omega_h\setminus\partial\Omega_h}, \mbox{ $i=1,2,3$},
\end{eqnarray*}
where $\tilde{K}$ is  a constant independent of $u$ and $h$. 
\end{Lemma}
\begin{proof} 
 See Appendix.
\end{proof}
\subsection{Discrete Helmholtz-Hodge decomposition}
Here is the  discrete Helmholtz-Hodge decomposition. 
\begin{Thm}\label{Projection}
For each function  $u:\Omega_h\to\R^3$, there exist unique functions $w:\Omega_h\to\R^3$ and  $\phi:\Omega_h\to\R$ such that 
\begin{eqnarray*}
&&\D\cdot w=0\quad\mbox{ on $\Omega_h$};\qquad
w+\D \phi =u\quad \mbox{ on $\Omega_h\setminus \partial\Omega_h$};\\
&&w=0\quad \mbox{ on $\partial\Omega_h$};\quad \sum_{x\in\Omega_h^{\circ j}}\phi(x)=0 \quad(j=1,\cdots,8),
\end{eqnarray*}
where $u$ does not necessarily need to banish on $\partial \Omega_h$. \end{Thm}
\begin{proof}
We modify the proof of Theorem 2.2 of \cite{Kuroki-Soga}. First, we note that any function $w:\Omega_h\to\R^3$ with $w|_{\partial\Omega_h}=0$ satisfies for each $j=1,\cdots,8$,
\begin{eqnarray}\label{trivial}
\sum_{x\in\Omega_h\cap G^j}\D\cdot w(x)=\sum_{i=1}^3\sum_{x\in\Omega_h\cap G^j}\frac{w_i(x+he^i)-w_i(x-he^i)}{2h}=0,
\end{eqnarray}
due to cancelation.  We label each point of $\Omega_h\setminus\partial\Omega_h$ and $\partial\Omega_h$ as 
$$ \Omega_h\setminus\partial\Omega_h=\{x^{1},x^{2},\ldots,x^{a}  \},\quad\partial\Omega_h=\{\bar{x}^1,\bar{x}^2,\ldots,\bar{x}^b\}.$$
Let $w,\phi$ be unknown functions to be determined. Introduce $y\in\R^{4a+b}$ and $\alpha\in\R^{4a+b+8}$ as  
\begin{eqnarray*}
y&=&\big(w_1(x^{1}),\ldots,w_1(x^{a}),w_2(x^{1}),\ldots,w_2(x^{a}),w_3(x^{1}),\ldots,w_3(x^{a}), \phi(x^1),\ldots,\phi(x^{a}),\\
&&\phi(\bar{x}^1),\ldots,\phi(\bar{x}^b)   \big),\\
\alpha&=&\big(0,\ldots,0,u_1(x^{1}),\ldots,u_1(x^{a}),u_2(x^{1}),\ldots,u_2(x^{a}),u_3(x^{1}),\ldots,u_3(x^{a}),\\
&&\quad 0,0,0,0,0,0,0,0        \big),
\end{eqnarray*}
where $\alpha$ has $a+b$ zero in front of $u_1(x^1)$.   Then, the equations $\D\cdot w=0$ on $\Omega_h$ , $w+\D \phi =u$ on $\Omega_h\setminus \partial\Omega_h$ with the zero mean constraint of $\phi$  give a $(4a+b+8)$-system of linear equations, which is denoted by $\tilde{A}y=\alpha$ with a $(4a+b+8)\times(4a+b)$-matrix $\tilde{A}$.  Due to \eqref{trivial}, we find eight trivial $0=0$ from  $\tilde{A}y=\alpha$. Therefore, $\tilde{A}y=\alpha$ can be deduced to be $Ay=\beta$ with a $(4a+b)\times(4a+b)$-matrix $A$ and $\beta\in \R^{4a+b}$. Note that $A$ is independent of $u$, and that $\beta=0$ if $u=0$. 

Our assertion holds, if $A$ is invertible. To prove invertibility of $A$, we show that $Ay=0$ if and only if $y=0$. There is at least one $y$ satisfying $Ay=0$. Then, we obtain  at least one pair $w,\phi$ satisfying 
\begin{eqnarray}\nonumber
&&\D\cdot w=0\quad\mbox{ on $\Omega_h$};\qquad
w+\D \phi =0\quad \mbox{ on $\Omega_h\setminus \partial\Omega_h$};\quad w=0\quad \mbox{ on $\partial\Omega_h$};\\\label{consthh}
&& \sum_{x\in\Omega_h^{\circ j}}\phi(x)=0 \quad(j=1,\cdots,8)
\end{eqnarray}
By Lemma \ref{vector-cal1}, we obtain
\begin{eqnarray*}
\sum_{x\in\Omega_h\setminus\partial\Omega_h}w(x)\cdot\D\phi(x)&=&\sum_{x\in\Omega_h}w(x)\cdot\D\phi(x)
=\sum_{x\in\Omega_h}(\D\cdot w(x))\phi(x)=0,\\
0&=&\sum_{x\in\Omega_h\setminus\partial\Omega_h}(w(x)+\D\phi(x))\cdot (w(x)+\D\phi(x))\\
&=&\sum_{x\in \Omega_h\setminus\partial\Omega_h}|w(x)|^2+\sum_{x\in \Omega_h\setminus\partial\Omega_h}|\D\phi(x)|^2.
\end{eqnarray*}
Therefore, $w=0$ on $\Omega_h$ and $\D\phi=0$ on  $\Omega_h\setminus\partial\Omega_h$. The latter equality implies that $\phi$ is constant on $\Omega_h\cap G^j$ for each  $j=1,\cdots,8$, and hence   \eqref{consthh} implies $\phi=0$.  Thus, $A$ is invertible. 

Suppose that there are two pairs $w,\phi$ and $\tilde{w},\tilde{\phi}$ which satisfy the assertion.  Then, we see that $w-\tilde{w}$, $\phi-\tilde{\phi}$ yields the unique trivial solution of $Ay=0$. Therefore, we conclude that $w=\tilde{w}$ and $\phi=\tilde{\phi}$.   
\end{proof}
\begin{Def}
Define the discrete Helmholtz-Hodge decomposition operator $P_h$ 
 for each function $u:\Omega_h\to\R^3$ as 
$$P_hu:=w,\,\,\,\mbox{$w$ is the one obtianed in Theorem \ref{Projection}. }$$
\end{Def}
\begin{Thm}\label{22estimate}
The following estimates hold for  the decomposition $u=P_h u+\D \phi$: 
\begin{eqnarray*}
&&\sum_{x\in \Omega_h}|P_h u(x)|^2\le \sum_{x\in\Omega_h\setminus \partial\Omega_h}|u(x)|^2,\quad \sum_{x\in \Omega_h\setminus \partial\Omega_h}|\D\phi(x)|^2\le \sum_{x\in\Omega_h\setminus \partial\Omega_h}|u(x)|^2,\\
&&\sum_{x\in \Omega_h^{\circ}}|\phi(x)|^2\le \tilde{A}^2\sum_{x\in\Omega_h\setminus \partial\Omega_h}|\D\phi(x)|^2\le \tilde{A}^2\sum_{x\in\Omega_h\setminus \partial\Omega_h}|u(x)|^2,
\end{eqnarray*}
where $\tilde{A}>0$ is the constant from the discrete Poincar\'e type inequality II. Furthermore, if $u=0$ on $(\Omega_h\setminus\Omega_h^{\circ})\cup\partial\Omega_h^\circ$, we have 
\begin{eqnarray}\label{242424242}
\sum_{x\in \Omega_h\setminus \partial\Omega_h}|u(x)-P_h u(x)|^2\le \tilde{A}^2\sum_{x\in\Omega_h^{\circ}}|\D\cdot u(x)|^2. 
\end{eqnarray}
\end{Thm}
\begin{proof} 
The assertion follows from reasoning similar to the proof of Theorem 2.3 of \cite{Kuroki-Soga} with Lemma \ref{vector-cal1} and  the discrete Poincar\'e type inequality II. For readers' convenience, we demonstrate  \eqref{242424242}:
\begin{eqnarray*}
&&\sum_{x\in \Omega_h\setminus \partial\Omega_h}|u(x)-P_h u(x)|^2=\sqrt{\sum_{x\in \Omega_h\setminus \partial\Omega_h}|u(x)-P_h u(x)|^2}\sqrt{\sum_{x\in \Omega_h\setminus \partial\Omega_h}|\D\phi(x)|^2}\\
&&\quad = \sum_{x\in \Omega_h\setminus \partial\Omega_h}(u(x)-P_h u(x))\cdot \D\phi(x)=\sum_{x\in \Omega_h\setminus \partial\Omega_h}u(x)\cdot \D\phi(x)\\
&&\quad =-\sum_{x\in \Omega_h\setminus \partial\Omega_h}
(\D\cdot u(x))\phi(x)=-\sum_{x\in \Omega_h^{\circ}}(\D\cdot u(x))\phi(x)\\
&&\quad \le\sqrt{\sum_{x\in \Omega_h^{\circ}}|\D\cdot u(x)|^2}
\sqrt{\sum_{x\in \Omega_h^{\circ}} |\phi(x)|^2 }
\le \sqrt{\sum_{x\in \Omega_h^{\circ}}|\D\cdot u(x)|^2} 
\sqrt{\tilde{A}^2\sum_{x\in \Omega_h\setminus \partial\Omega_h}|\D\phi(x)|^2}.
\end{eqnarray*}
\end{proof}
\noindent Note that the values of $\D\phi$ on $\partial\Omega_h$ are out of control in the decomposition, which requires the discrete Poincar\'e type inequalities not to contain those values; We will discuss an inequality corresponding to \eqref{242424242} without the condition  $u=0$ on $(\Omega_h\setminus\Omega_h^{\circ})\cup\partial\Omega_h^\circ$ in Section 3. 

\subsection{Discrete Navier-Stokes equations}

Let $\tau>0$ be the time-discretization parameter and let $T_\tau\in\N$ be such that $T\in[\tau T_\tau,\tau T_\tau+\tau)$. For initial data  $v^0=(v^0_1,v^0_2,v^0_3)\in L^2_{0,\sigma}(\Omega)$ and the external force   $f=(f_1,f_2,f_3)\in L^2_{loc}([0,\infty); L^2(\Omega)^3)$, introduce $\tilde{u}^0=(\tilde{u}_1^0,\tilde{u}_2^0,\tilde{u}_3^0):\Omega_h\to\R^3$, $f^{n}=(f^{n}_1,f^{n}_2,f^{n}_3):\Omega_h\to\R^3$, $n=0,1,\cdots$ as  
\begin{eqnarray*}
\tilde{u}^0_i(x)&=&\left\{
\begin{array}{lll}
&\dis h^{-3}\int_{{C_h(x)}}v^0_i(y)dy,\quad x\in\Omega_h\setminus\partial\Omega_h,
\medskip\\
&0\mbox{\quad\quad\quad\,\,\,\quad\,\,\, otherwise},
\end{array}
\right. \\\\
\\
f^{n}_i(x) &:=&\tau^{-1}h^{-3}\int_{\tau n}^{\tau (n+1)}\int_{{C_h(x)}}f_i(t,y)dydt, \quad   x\in\Omega_h.
\end{eqnarray*}
 We define functions   $\tu^{n+1}(\cdot)=(\tu^{n+1}_1(\cdot),\tu^{n+1}_2(\cdot),\tu_3^{n+1}(\cdot)):\Omega_h\to\R^3$, $n=0,1,\ldots,T_\tau-1$ and $u^n(\cdot)=(u^n_1(\cdot),u^n_2(\cdot),u_3^n(\cdot)):\Omega_h\to\R^3$, $n=0,1,\ldots,T_\tau$  in the following manner ($\tu^0$ is already defined above):     
\begin{eqnarray}\label{initial-1}
u^0&=&P_h \tilde{u}^0,\\\label{fractional-1}
\frac{\tu^{n+1}(x)-u^n(x)}{\tau}&=& -\frac{1}{2}\sum_{j=1}^3 \Big(u^n_j(x-he^j)D_j\tu^{n+1}(x-he^j)\\\nonumber
&&+u^n_j(x+he^j)D_j\tu^{n+1}(x+he^j)\Big)\\\nonumber
&&+\sum_{j=1}^3 D_j^2\tu^{n+1} (x)+f^n(x),\quad x\in\Omega_h\setminus\partial\Omega_h,\\\label{fractional-2}
\tu^{n+1}(x)&=&0, \quad x\in\partial\Omega_h,\\\label{fractional-3}
u^{n+1}&=&P_h\tu^{n+1},
\end{eqnarray}
where \eqref{initial-1}-\eqref{fractional-3} are recurrence equations in the implicit form and called the discrete Navier-Stokes equations. 

For functions $u,w:\Omega_h\to\R^3$ or $\R$, we define the discrete $L^2$-inner product and norm as 
$$(u,v)_{\Omega_h}:=\sum_{x\in\Omega_h}u(x)\cdot w(x) h^3,\quad \norm u \norm_{\Omega_h}:=\sqrt{(u,u)_{\Omega_h}}.$$
The next theorem states unconditional solvability of the implicit equations  \eqref{initial-1}-\eqref{fractional-2}.  
\begin{Thm}\label{implicit}
Suppose that $u^n:\Omega_h\to\R^3$ satisfies $\D\cdot u^n=0$ on $\Omega_h\setminus\partial\Omega_h$ and $u^n=0$ on $\partial\Omega_h$ for some $n$. Then,  the equation \eqref{fractional-1}-\eqref{fractional-2} is uniquely solvable with respect to $\tu^{n+1}$ for any mesh size $h,\tau$. 
\end{Thm}
\begin{proof}
Although our proof is essentially the same as  the proof of Theorem 3.1 of \cite{Kuroki-Soga}, we demonstrate some calculation. We label the elements of $\Omega_h\setminus\partial\Omega_h$ as $x^{1},x^{2},\ldots,x^{a}$. Introduce $y,\alpha\in\R^{3a}$ as  
\begin{eqnarray*}
y&:=&\big(\tu_1^{n+1}(x^{1}),\ldots,\tu_1^{n+1}(x^{a}),\tu_2^{n+1}(x^{1}),\ldots,\tu_2^{n+1}(x^{a}),\tu_3^{n+1}(x^{1}),\ldots,\tu_3^{n+1}(x^{a})   \big),\\
\alpha&:=&\big(u_1^{n}(x^{1})+\tau f^{n}_1(x^1),\ldots,u_1^{n}(x^{a})+\tau f^{n}_1(x^a),u_2^{n}(x^{1})+\tau f^{n}_2(x^1),\\
&&\ldots,u_2^{n}(x^{a})+\tau f^{n}_2(x^a),u_3^{n}(x^{1})+\tau f^{n}_3(x^1),\ldots,u_3^{n}(x^{a}) +\tau f^{n}_3(x^a)  \big).
\end{eqnarray*}
Then,  \eqref{fractional-1}-\eqref{fractional-2} are re-written as the linear equations $A(u^n;h,\tau)y=\alpha$ 
with  a $3a\times3a$-matrix $A(u^n;h,\tau)$  depending on $u^n,h,\tau$. 

We prove that the matrix $A(u^n;h,\tau)$ is invertible if $u^n$ satisfies $\D\cdot u^n=0$ in $\Omega_h\setminus\partial\Omega_h$. It is enough to check that  $A(u^n;h,\tau)\tilde{y}=0$ has the unique solution $\tilde{y}=0$.  Let $\tilde{y}=y_0$ be a solution of $A(u^n;h,\tau)\tilde{y}=0$. Then, there exists at least one function $\tu^{n+1}:\Omega_h\to\R^3$ with $\tilde{u}^{n+1}|_{\partial\Omega_h}=0$ such that    
\begin{eqnarray*}
\tu^{n+1}(x)&=&-\frac{\tau}{2}\sum_{j=1}^3 \Big(u^n_j(x-he^j)D_j\tu^{n+1}(x-he^j)+u^n_j(x+he^j)D_j\tu^{n+1}(x+he^j)\Big)\\\nonumber
&&+\tau\sum_{j=1}^3D_j^2\tu^{n+1} (x),\quad x\in\Omega_h\setminus\partial\Omega_h.
\end{eqnarray*}
Then, we have
\begin{eqnarray*}
(\tu^{n+1},\tu^{n+1})_{\Omega_h}&=&\norm \tu^{n+1}\norm_{\Omega_h}^2\\
&=&-\frac{\tau}{2}\sum_{i,j=1}^3 \sum_{x\in\Omega_h\setminus\partial\Omega_h}\Big(u^n_j(x-he^j)D_j\tu^{n+1}_i(x-he^j)\\
&& + u^n_j(x+he^j)D_j\tu^{n+1}_i(x+he^j)\Big)\tu^{n+1}_i(x)h^3\\
 &&+\tau\sum_{i,j=1}^3 \sum_{x\in\Omega_h\setminus\partial\Omega_h} D_j^2\tu_i^{n+1}(x)\tu_i^{n+1}(x) h^3.
\end{eqnarray*}
The  above  two  summations are denoted by (i), (ii), respectively. With the zero boundary condition of $\tu^{n+1}$, we see that  
 \begin{eqnarray*}
{\rm(i)}& =&\sum_{i,j=1}^3 \sum_{x\in\Omega_h\setminus\partial\Omega_h}\Big(u^n_j(x-he^j)\frac{\tu^{n+1}_i(x)-\tu^{n+1}_i(x-2he^j)}{2h}\\
&&+u^n_j(x+he^j)\frac{\tu^{n+1}_i(x+2he^j)-\tu^{n+1}_i(x)}{2h} \Big)\tu^{n+1}_i(x)h^3\\
&=&\sum_{i,j=1}^3 \sum_{x\in\Omega_h\setminus\partial\Omega_h}
-\frac{u^n_j(x+he^j)-u^n_j(x-he^j)}{2h}\tu^{n+1}_i(x)^2h^3\\
&&+\sum_{i,j=1}^3 \sum_{x\in\Omega_h}
\frac{1}{2h}u^n_j(x+he^j)\tu^{n+1}_i(x+2he^j)\tu^{n+1}_i(x)h^3\\
&&-\sum_{i,j=1}^3 \sum_{x\in\Omega_h} \frac{1}{2h}u^n_j(x-he^j)\tu^{n+1}_i(x-2he^j)\tu^{n+1}_i(x)h^3.
\end{eqnarray*} 
Shifting $x$ to $x\mp 2he^j$ in the last summation, we obtain 
 \begin{eqnarray*}
{\rm(i)}=-\sum_{x\in\Omega_h\setminus\partial\Omega_h}\big(\D\cdot u^n(x)\big)|\tu^{n+1}(x)|^2h^3.
\end{eqnarray*} 
Similarly, we see that 
\begin{eqnarray*}
\rm{(ii)}&=&  \sum_{i,j=1}^3 \sum_{x\in\Omega_h\setminus\partial\Omega_h}
\frac{\tu_i^{n+1} (x+he^j)-2\tu_i^{n+1} (x)+\tu_i^{n+1} (x-he^j)}{h^2}\tu^{n+1}_i(x)h^3\\
&=& \sum_{i,j=1}^3 \sum_{x\in\Omega_h\setminus\partial\Omega_h}
\frac{\tu_i^{n+1} (x+he^j)-\tu_i^{n+1} (x)}{h^2}\tu^{n+1}_i(x)h^3\\
&&- \sum_{i,j=1}^3 \sum_{x\in\Omega_h\setminus\partial\Omega_h}\frac{\tu_i^{n+1} (x)-\tu_i^{n+1} (x-he^j)}{h^2}\tu^{n+1}_i(x)h^3\\
&=& \sum_{i,j=1}^3 \sum_{x\in\Omega_h}
\frac{\tu_i^{n+1} (x+he^j)-\tu_i^{n+1} (x)}{h^2}\tu^{n+1}_i(x)h^3\\
&&- \sum_{i,j=1}^3 \sum_{x\in\Omega_h}\frac{\tu_i^{n+1} (x+he^j)-\tu_i^{n+1} (x)}{h^2}\tu^{n+1}_i(x+he^j)h^3\\
&=&- \sum_{j=1}^3\norm D^+_j\tu^{n+1}\norm_{\Omega_h}^2\le0. 
\end{eqnarray*} 
Hence, the discrete divergence free constraint of $u^{n}$ implies 
$$\norm \tu^{n+1}\norm_{\Omega_h}^2+\tau\sum_{j=1}^3\norm D^+_j\tu^{n+1}\norm_{\Omega_h}^2=0.$$
Thus, we conclude that $\tu^{n+1}=0$ and $y_0=0$. 
\end{proof}
\begin{Thm}\label{L2-estimate}
For any  $h,\tau>0$,  the discrete problem \eqref{initial-1}-\eqref{fractional-3} is uniquely solvable for $n=0,1,\ldots,T_\tau-1$ with  the following estimates:  
\begin{eqnarray} \label{410}
\norm u^0\norm_{\Omega_h}&\le&\norm \tilde{u}^0\norm_{\Omega_h}\le \norm v^0\norm_{L^2(\Omega)^3},\\\label{4100}
\sum_{m=0}^n\norm f^{m}\norm_{\Omega_h}^2\tau &\le& \norm f\norm_{L^2([0,\tau (n+1)];L^2(\Omega)^3)}^2\le \norm f\norm_{L^2([0,T];L^2(\Omega)^3)}^2,\\\label{411}
\norm \tu^{n+1}\norm_{\Omega_h}&\le& \norm u^n\norm_{\Omega_h}+  \norm f^n\norm_{\Omega_h}\tau,\\\label{412}
\norm u^{n+1}\norm_{\Omega_h}&\le& \norm u^0\norm_{\Omega_h}+\sum_{m=0}^{T_\tau-1}\norm f^{m}\norm_{\Omega_h}\tau\\\nonumber   
&\le& \norm v^0\norm_{L^2(\Omega)^3}+\sqrt{T}\norm f\norm_{L^2([0,T];L^2(\Omega)^3)},\\\label{413}
\norm u^{n+1}\norm_{\Omega_h}^2&\le& \norm u^0\norm_{\Omega_h}^2-\sum_{m=0}^n  \Big(  \sum_{j=1}^3\norm D^+_j\tu^{m+1}\norm_{\Omega_h}^2 \Big)\tau \\\nonumber
 && +2\sum_{m=0}^n\norm u^m\norm_{\Omega_h} 
 \norm f^{m} \norm_{\Omega_h}\tau 
 +\sum_{m=0}^{n}\norm f^{m}\norm_{\Omega_h}^2\tau^2. 
\end{eqnarray}
\end{Thm}
\begin{proof}
We may follow the proof of Theorem 4.1 of \cite{Kuroki-Soga}.
\end{proof}
 Theorem \ref{L2-estimate} implies convergence of the discrete solution to a Leray-Hopf weak solution (up to a subsequence). Set $\delta:=(h,\tau)$. For the solution $u^n,\tu^{n+1}$ of  \eqref{initial-1}-\eqref{fractional-3}, define the step functions  $u_\delta, \tu_\delta,w^i_\delta:[0,T]\times\Omega\to\R^3$, $i=1,2,3$ as 
\begin{eqnarray}\label{step1}
u_\delta(t,x)&:=&\left\{
\begin{array}{lll}
&u^{n}(y)\mbox{\quad\quad\ \,\,\,\, for $t\in[n\tau,n\tau+\tau)$, $x\in {C_h^+(y)}$, $y\in \Omega_h$},
\medskip\\
&0\mbox{\quad\quad\quad\,\,\,\quad\,\,\, otherwise},
\end{array}
\right. \\ \label{step2}
\tu_\delta(t,x)&:=&\left\{
\begin{array}{lll}
&\tu^{n+1}(y)\mbox{\quad\quad\ for $t\in[n\tau,n\tau+\tau)$, $x\in {C_h^+(y)}$, $y\in \Omega_h$},
\medskip\\
&0\mbox{\quad\quad\quad\,\,\,\quad\,\,\, otherwise},
\end{array}
\right. \\\label{step3}
w^i_\delta(t,x)&:=&\left\{
\begin{array}{lll}
&D_i^+\tu^{n+1}(y)\mbox{\quad for $t\in[n\tau,n\tau+\tau)$, $x\in {C_h^+(y)}$, $y\in \Omega_h$},
\medskip\\
&0\mbox{\quad\quad\quad\quad\quad\, otherwise}, 
\end{array}
\right.
\end{eqnarray}
where $n=0,1,\ldots,T_\tau-1$ and the notation $C^+_h(y)$ is seen in Lemma \ref{Lip-interpolation}.  
In the rest of our argument, the statement ``there exists a sequence $\delta\to0$ ...'' means ``there exists a sequence $\delta_l=(h_l,\tau_l)$ with $h_l,\tau_l\searrow0$ as $l\to\infty$ ...''. 
\begin{Thm}\label{weak convergence}
There exists a sequence $\delta\to0$ and a function $v\in L^2([0,T];H^1_{0,\sigma}(\Omega))$ for which  the following weak convergence holds: 
\begin{eqnarray}\label{511}
&&u_\delta \wto v \mbox{ \quad in $L^2([0,T];L^2(\Omega)^3)$ as $\delta\to0$},\\\label{512}
&&\tu_\delta \wto v \mbox{ \quad in $L^2([0,T];L^2(\Omega)^3)$ as $\delta\to0$},\\\label{523}
&&w^i_\delta \wto \partial_{x_i} v \mbox{ \quad in $L^2([0,T];L^2(\Omega)^3)$ as $\delta\to0$ ($i=1,2,3$)}.
\end{eqnarray}  
\end{Thm}
\begin{proof}
We may follow the proof of Theorem 5.1 of \cite{Kuroki-Soga} with Theorem \ref{L2-estimate} and Lemma \ref{Lip-interpolation}. 
\end{proof}
\begin{Thm}\label{strong-convergence}
Take $\delta=(h,\tau)\to0$ under the condition $ h^{4-\alpha}\le \tau$, where  $\alpha\in(0,2]$ is any constant. Then, the sequence $\{\tu_\delta\}$, which satisfies \eqref{512}, converges strongly to $v$  in $L^2([0,T];L^2(\Omega)^3)$ as $\delta\to0$.    
\end{Thm}
\begin{proof}
We may follow the proofs of Lemma 6.1 and Theorem 6.2 of \cite{Kuroki-Soga}, where we slightly change $Q_h$ in $\nnorm\cdot\nnorm_{\rm op}$ to be  
$(Q_h\phi)_j:=\phi_j-\frac{h^2}{3!}\frac{\partial^2\phi_j}{\partial x_j^2}$ (note that we use the central difference for the discrete divergence). 
\end{proof}
\begin{Thm}\label{Leray-Hopf}
The limit function $v$ of $\{u_\delta\}$ and $\{\tilde{u}_\delta\}$ derived under  the condition $ h^{4-\alpha}\le \tau$ with $\alpha\in(0,2]$ is a Leray-Hopf weak solution of (\ref{NS}).
\end{Thm}
\begin{proof}
We may follow the proof of Theorem 7.1 of \cite{Kuroki-Soga}. 
\end{proof}
\subsection{Time-global solvability}
We sharpen Theorem \ref{L2-estimate} by taking the dissipative effect of $D_j^2$ into account and prove time-global solvability of the discrete Navier-Stokes equations under the assumption that there exists a constant $\alpha\ge0$ for which the external force $f\in L^2_{\rm loc}([0,\infty);L^2(\Omega)^3)$ satisfies 
$$\norm f \norm_{L^2([n-1,n];L^2(\Omega)^3)}\le \alpha\mbox{\quad for all $n\in\N$}.$$
A typical example of such $f$ is time-periodic one, which will be discussed in Section 4.  

Take $\tau=1/T_1$ with $T_1\in\N$. Define the set $\tilde{U}_R$ of initial data $\tilde{u}^0$ as 
$$\tilde{U}_R:=\{\tilde{u}:\Omega_h\to\R^3 \,|\,\,\,\,\norm \tu\norm_{\Omega_h}\le R, \quad \tilde{u}|_{\partial\Omega_h}=0  \},\quad R\ge0,$$ 
and the constant $R_0(\Omega,f)\ge0$ as 
\begin{eqnarray*}
R_0(\Omega,f):=\frac{1}{1-e^{-A^{-2}}}\Big(\frac{1-e^{-2A^{-2}}}{2A^{-2}}\Big)^\2\alpha,
\end{eqnarray*}
where $A>0$ is the constant from the discrete Poincar\'e type inequality I. Note that $A$ depends only on the diameter of $\Omega$.  
\begin{Thm}\label{nonperiodic-time-global}
Let $R\ge R_0(\Omega,f)$. Then, for each $\tu^0\in \tilde{U}_R$, the discrete Navier-Stokes equations \eqref{initial-1}-\eqref{fractional-3} are solvable for all $n\in\N$ and the solution satisfies $\tu^{mT_1}\in \tilde{U}_R$ for all $m\in\N$. 
\end{Thm}
\begin{proof}
 It follows from the equalities for (i), (ii) in the proof of Theorem \ref{implicit} and the discrete Poincar\'e type inequality I that the inner product of \eqref{fractional-1} and $\tilde{u}^{n+1}$ yields
\begin{eqnarray*}
\norm \tu^{n+1}\norm_{\Omega_h}&\le& \norm u^n\norm_{\Omega_h}+\norm f^n\norm_{\Omega_h}\tau-3(A^{-1})^2 \norm \tu^{n+1}\norm_{\Omega_h}\tau.
\end{eqnarray*}
Hence, we have 
\begin{eqnarray*}
\norm \tu^{n+1}\norm_{\Omega_h}&\le& \frac{1}{1+3(A^{-1})^2\tau}\norm u^n\norm_{\Omega_h}+\frac{1}{1+3(A^{-1})^2\tau}\norm f^n\norm_{\Omega_h}\tau\\
&\le&\Big(\frac{1}{1+3(A^{-1})^2\tau}\Big)^{n+1}\norm \tu^0 \norm_{\Omega_h} +\sum_{m=0}^{n}\Big(\frac{1}{1+3(A^{-1})^2\tau}\Big)^{n+1-m}\norm f^m\norm_{\Omega_h}\tau\\
&\le&e^{-(A^{-1})^2(n+1)\tau}\norm \tu^0\norm_{\Omega_h}+\sum_{m=0}^ne^{-(A^{-1})^2(n+1-m)\tau}\norm f^m\norm_{\Omega_h}\tau.
\end{eqnarray*}
Therefore, we obtain 
\begin{eqnarray*}
\norm \tu^{T_1}\norm_{\Omega_h}&\le& e^{-(A^{-1})^2}\norm \tu^0\norm_{\Omega_h}
+\left(\sum_{m=0}^{T_1-1}e^{-2(A^{-1})^2(1-m\tau)}\tau\right)^\2\left(\sum_{m=0}^{T_1-1}\norm f^m\norm_{\Omega_h}^2\tau\right)^\2\\
&\le&e^{-(A^{-1})^2}\norm \tu^0\norm_{\Omega_h}
+\Big(\int_0^1e^{-2(A^{-1})^2(1-t)} dt \Big)^\2\norm f\norm_{L^2([0,1];L^2(\Omega)^3)}
\\
&\le&e^{-(A^{-1})^2}R
+\Big(\frac{1-e^{-2(A^{-1})^2}}{2(A^{-1})^2}\Big)^\2\alpha.
\end{eqnarray*}
We see that 
\begin{eqnarray*}
&&R\ge R_0(\Omega,f)=\frac{1}{1-e^{-A^{-2}}}\Big(\frac{1-e^{-2A^{-2}}}{2A^{-2}}\Big)^\2\alpha
\quad\\
&&\qquad\qquad\qquad\qquad\qquad  \Leftrightarrow\quad  R\ge e^{-A^{-2}}R
+\Big(\frac{1-e^{-2A^{-2}}}{2A^{-2}}\Big)^\2\alpha. 
\end{eqnarray*}
Thus, if $\tu^0$ satisfies $\norm \tu^0\norm_{\Omega_h}\le R$ with  $R\ge R_0(\Omega,f)$, we have $\tilde{u}^{T_1}\in\tilde{U}_R$. 
We, then, repeat the same estimate to obtain $\tilde{u}^{2T_1}\in\tilde{U}_R$, $\tilde{u}^{3T_1}\in\tilde{U}_R$ and so on. 
\end{proof}
\begin{Thm}\label{time-global-Leray-Hopf}
Let $u_\delta,\tilde{u}_\delta:[0,\infty)\times\Omega\to\R^3$ be the step functions derived from \eqref{step1}, \eqref{step2} and Theorem \ref{nonperiodic-time-global}. There exists a sequence $\delta\to0$ for which $\{\tilde{u}_\delta\}$, $\{u_\delta\}$ tend to  a time-global  Leray-Hopf weak solution $v$ in $L_{\rm loc}^2([0,\infty),L^2(\Omega)^3)$, i.e., $\{\tilde{u}_\delta|_{t\in[0,T]}\}$ (resp. $\{u_\delta|_{t\in[0,T]}\}$) converges strongly (resp. weakly) to $v|_{t\in[0,T]}$ in $L^2([0,T],L^2(\Omega)^3)$ for any fixed $T>0$ as $\delta\to0$. 
\end{Thm}
\begin{proof}
Standard  Cantor's diagonal argument yields the assertion. In fact, for each $m\in\N$, Theorem \ref{Leray-Hopf} implies that there exists a sequence $\{\delta_{m,l}\}_{l\in\N}$ with $\delta_{m,l}\to0$ as $l\to\infty$ such that $\{u_{\delta_{m,l}}|_{t\in[0,m]}\}_{l\in\N}$, $\{\tilde{u}_{\delta_{m,l}}|_{t\in[0,m]}\}_{l\in\N}$ converge to a Leray-Hopf weak solution defined in $[0,m]\times\Omega$.  
Then, we may subtract a subsequence $\{\delta_{m+1,l}\}_{l\in\N}$ from $\{\delta_{m,l}\}_{l\in\N}$ such that $\{u_{\delta_{m+1,l}}|_{t\in[0,m+1]}\}_{l\in\N}$, $\{\tilde{u}_{\delta_{m+1,l}}|_{t\in[0,m+1]}\}_{l\in\N}$ converge to a Leray-Hopf weak solution defined in $[0,m+1]\times\Omega$. Repeating this process for $m=1,2,\ldots$ and taking the sequence $\{\delta_{m,m}\}_{m\in\N}$, we obtain our assertion.
\end{proof}

\setcounter{section}{2}
\setcounter{equation}{0}
\section{Error estimate in $C^3$-class}
We give an error estimate for our projection method, {\it supposing that the  external force $f$ is smooth and that the limit $v$ of Theorem \ref{Leray-Hopf} belongs to the $C^3([0,T]\times\bar{\Omega})$-class with the pressure $p\in C^2([0,T]\times\bar{\Omega})$.} Note that a Leray-Hopf weak solution is smooth within a certain time interval, provided initial data and $\partial\Omega$ are smooth enough. The argument below itself does not require smoothness of $\partial\Omega$, and we proceed with the Lipschitz regularity of $\partial\Omega$ (we do not discuss if there is a special situation where a Lipschitz domain $\Omega$ yields a $C^3([0,T]\times\bar{\Omega})$-solution). 

Difficulty here is that $\partial\Omega_h$ is not contained in $\partial\Omega$; Hence, the exact solution $v$ does not satisfy the zero boundary condition on $\partial\Omega_h$; The calculus on $\Omega_h$ applied to $v|_{\Omega_h}$ leaves reminder terms coming from $v|_{\partial\Omega_h}=O(h)$.  Careful estimates of such reminder terms are  necessary. For this purpose, we assume that our Lipschitz domain $\Omega$ satisfies the following property:
\medskip

\noindent{\bf Condition A.}  
{\it There exist a constant $s_0>0$  and a family $\{V_k\}_{k=1,\cdots,K}$ of open subset of planes in $\R^3$ such that  
\begin{itemize}
\item Each $V_k$ is contained in $\Omega$ and has a normal vector $\omega_k\in \{\pm e^j\}_{j=1,2,3}$ such that for each $y\in V_k$,
\begin{eqnarray*}
\{y+s\omega_k\,|\,s\ge0\}\cap\partial\Omega=\mbox{\rm singleton}=\{ y+\varphi_k(y)\omega_k \},
\end{eqnarray*}
where $\varphi_k(y):V_k\to\R_{>0}$ is seen as a hight function between  $V_k$  and $\partial\Omega$,
\item $\varphi_k(y)\ge s_0$ for all $0\le k\le K$ and $y\in V_k$,
\item $\displaystyle\bigcup_{0\le k\le K} \{y+\varphi_k(y)\omega_k\,|\,y\in V_k  \}=\partial\Omega$.
\end{itemize}}

\medskip

\noindent Note that Condition A is fulfilled if $\partial\Omega$ is smooth; $\Omega$ being rectangular with $\partial\Omega$ orthogonal to $e^1$, $e^2$ or $e^3$ fails to satisfy Condition A (the edge is left over in the last condition), but we may directly deal with such an $\Omega$ through the reasoning in Subsection 3.1.    

   Our goal is to prove the next Theorem.
\begin{Thm}\label{error-estimate}
Suppose that Condition A holds. Suppose also that (\ref{NS}) with a smooth external force $f$ possesses the solution $(v,p)$ such that $v\in C^3([0,T]\times\bar{\Omega})$ and $p\in C^2([0,T]\times\bar{\Omega})$. Then, the solution $u^n,\tilde{u}^n$ to the discrete problem under the scaling condition  $\tau=\theta h^{\frac{3}{4}}$, $\theta\in[\theta_0, \theta_1]$ with fixed  constants $\theta_0,\theta_1>0$ satisfies 
$$\max_{0\le n\le T_\tau}\norm u^n-v(\tau n,\cdot)\norm_{\Omega_h}\le \beta^\ast h^{\frac{1}{4}},\quad \max_{0\le n\le T_\tau}\norm \tilde{u}^n-v(\tau n,\cdot)\norm_{\Omega_h}\le \beta^\ast h^{\frac{1}{4}}\mbox{\quad as $\tau,h\to0$},$$
where $\beta^\ast>0$
 is a constant independent of $\tau,h$ and $\theta$. 
 \end{Thm}
Our strategy is the following: Let $v\in C^3([0,T]\times\bar{\Omega})$,  $p\in C^2([0,T]\times\bar{\Omega})$ satisfy  \eqref{NS} in the sense of classical solutions. For $n=0,1,\cdots,T_\tau$, define 
\begin{eqnarray*}
v^n(\cdot):=v(\tau n,\cdot),\quad p^n(\cdot):=p(\tau n,\cdot).
\end{eqnarray*}    
For each $x\in\Omega_h\setminus\partial\Omega_h$, set  
\begin{eqnarray*}
R^n(x)&:=&v^{n+1}(x)-\Big\{v^{n}(x)-\frac{\tau}{2}\sum_{j=1}^3\big(v^n_j(x-he^j)D_jv^{n+1}(x-he^j)\\
&&\quad +v^n_j(x+he^j)D_jv^{n+1}(x+he^j)\big) +\tau\sum_{j=1}^3D^2_jv^{n+1}(x)+\tau f^{n}(x)\Big\}.
\end{eqnarray*}
It follows from the Taylor expansion that 
\begin{eqnarray*}
&&v^{n+1}(x)-v^n(x)=\tau\partial_tv^n(x)+O(\tau^2);\\
&&\frac{\tau}{2}\sum_{j=1}^3\big(v^n_j(x-he^j)D_jv^{n+1}(x-he^j)
+v^n_j(x+he^j)D_jv^{n+1}(x+he^j)\big)\\
&&\quad=\frac{\tau}{2}\sum_{j=1}^3\big(v^n_j(x-he^j)D_jv^{n}(x-he^j)
+v^n_j(x+he^j)D_jv^{n}(x+he^j)\big)+O(\tau^2)\\
&&\quad=\frac{\tau}{2}\sum_{j=1}^3\Big\{ \Big(v^n_j(x)-\partial_{x_j}v_j(x)h+O(h^2)\Big)\Big(\partial_{x_j}v^{n}(x)-\frac{1}{2}\partial^2_{x_j}v^n(x)\cdot 2h+O(h^2)\Big)\\
&&\qquad +\Big(v^n_j(x)+\partial_{x_j}v_j(x)h+O(h^2)\Big)\Big(\partial_{x_j}v^{n}(x)+\frac{1}{2}\partial^2_{x_j}v^n(x)\cdot 2h+O(h^2)\Big)\Big\}+O(\tau^2)\\
&&\quad =\tau\sum_{j=1}^3v_j^n(x)\partial_{x_j}v^n(x)+O(\tau h^2) +O(\tau^2);\\
&&\tau\sum_{j=1}^3D^2_jv^{n+1}(x)=\tau\sum_{j=1}^3D^2_jv^{n}(x)+O(\tau^2)=\tau\sum_{j=1}^3\partial_{x_j}^2v^{n}(x)+O(\tau h)+O(\tau^2);\\
&&\tau f^n(x)=\tau f(\tau n,x)+O(\tau^2)+O(\tau h). 
\end{eqnarray*}
Hence, the exact Navier-Stokes equations imply that 
\begin{eqnarray*}
&&R^n(x)=-\tau \nabla p^n(x)+O(\tau h)+O(\tau^2)=-\tau \D p^n(x)+O(\tau h)+O(\tau^2)\quad \mbox{ on $\Omega_h\setminus\partial\Omega_h$}.
\end{eqnarray*}
 Let $u^n,\tilde{u}^{n+1}$ be the solution of  \eqref{initial-1}-\eqref{fractional-3} with $\tilde{u}^0$ given by $v^0=v(0,\cdot)$. Set 
$$b^n:=u^n-v^n,\quad \tilde{b}^{n}:=\tilde{u}^{n}-v^{n}.$$
Then, we have for $x\in\Omega_h\setminus\partial\Omega_h$,
\begin{eqnarray}\label{bbbbb}
&&\tilde{b}^{n+1}(x)
=b^n(x) \\\nonumber
&&\quad -\underline{ \frac{\tau}{2}\sum_{j=1}^3   
\Big(u^n_j(x-he^j)D_j\tilde{b}^{n+1}(x-he^j)+u^n_j(x+he^j)D_j\tilde{b}^{n+1}(x+he^j)\Big)}_{\rm (i)} \\\nonumber 
&&\quad-\underline{\frac{\tau}{2}\sum_{j=1}^3 \Big(b^n_j(x-he^j)D_jv^{n+1}(x-he^j)+b^n_j(x+he^j)D_jv^{n+1}(x+he^j) \Big)}_{\rm (ii)}\\\nonumber
&&\quad+\underline{\tau\sum_{j=1}^3D^2_j\tilde{b}^{n+1}(x)}_{\rm (iii)} -R^n(x). 
\end{eqnarray}
In order to have a recurrence inequality of the norm of $\tilde{b}^n$ with respect to $n$ from \eqref{bbbbb}, we need the estimate  
\begin{eqnarray}\label{ideaidea}
\norm b^n\norm_{\Omegain}&=&\norm P_h(\tilde{u}^n-v^n)+P_h v^n-v^n\norm_{\Omegain} \\\nonumber
&\le& \norm \tilde{b}^n\norm_{\Omegain}+ \norm v^n-P_hv^n\norm_{\Omegain}.
\end{eqnarray}
The term $ \norm v^n-P_hv^n\norm_{\Omegain}$ must be treated as small increment of error within $\tau$ even though it does not contain $\tau$; namely, we have to  take $\tau$ out of this term with an appropriate scaling condition. 

We remark that {\it in the rest of this section, the discrete differential operators $D_j^+,D_j,\D$ operate on $v^n|_{\Omega_h}$ and $p^n|_{\Omega_h}$ in \eqref{bbbbb} without the $0$-extension outside $\Omega_h$, while $P_h$ operates on $v^n|_{\Omega_h}$ with the $0$-extension outside $\Omega_h$.} 

We will demonstrate $L^2$-estimates of \eqref{bbbbb} and \eqref{ideaidea}, where {\it we must take  care of remainder terms of ``summation by parts'' coming from $b^n|_{\partial\Omega_h}=-v^n|_{\partial\Omega_h}\neq0$ and $\tilde{b}^n|_{\partial\Omega_h}=-v^n|_{\partial\Omega_h}\neq0$.}   For this purpose, we prepare several lemmas below. Note that, if $v$ and $f$ have more regularity, $O(\tau h)$ in $R^n(x)$ can be $O(\tau h^2)$, which is essential in the problem with the periodic boundary conditions for a shaper error estimate (see Section 5 and \cite{Chorin}). In the Dirichlet problem, however, $O(\tau h^2)$ is not necessary in $R^n(x)$, because \eqref{ideaidea} gives  lower order error.  
\subsection{Estimates on  boundary}

We show that $\norm v^n-P_hv^n\norm_{\Omegain}$ is of $O(h)$ at best in general. Then, we must take $\tau$ out of  $O(h)$ with a scaling condition in accordance with the other remainder terms. We will see that the appropriate scaling is $\tau=O(h^{\frac{3}{4}})$, which implies $O(h)=\tau O(h^\frac{1}{4})$. 
 One can say that  {\it convergence rate of the fully discrete projection method with the Dirichlet boundary is governed by the estimate of  $\norm v^n-P_hv^n\norm_{\Omegain}$ in  \eqref{bbbbb} and \eqref{ideaidea}.} 
 Our argument requires several estimates on/near the  boundary of $\Omega_h$, which is reminiscent of the construction of the trace operator on $H^1(\Omega)$.      

It is useful to observe that we have a constant $\beta>0$ such that 
$$(\sharp\Gamma_h^{j\pm})h^2,(\sharp\tilde{\Gamma}_h^{j\pm})h^2,(\sharp\partial\Omega_h)h^2\le\beta\mbox{\quad as $h\to0$}.$$
Let $\chi_A$ be the indicator function supported by $A$. We sometimes use calculation like  
\begin{eqnarray}\label{jiro1}
&&\sum_{x\in\partial\Omega_h}|u(x)|h^3\le \beta\max_{x\in\partial\Omega_h}|u(x)|h,\\\label{jiro2}
&&\sum_{x\in\partial\Omega_h}|u(x)|h^3=\sum_{x\in\Omega_h}\chi_{\partial\Omega_h}(x)|u(x)|h^3\le\norm \chi_{\partial\Omega_h}\norm_{\Omega_h}\norm u\norm_{\Omega_h}=O(h^\2)\norm u\norm_{\Omega_h}.
\end{eqnarray}
\indent We prepare several lemmas. 
\begin{Lemma}\label{lem-divdivdiv}
There exist constants $a,b>0$ depending only on $\Omega$ for which each function $u:\Omega_h\to\R^3$ satisfies the estimate
\begin{eqnarray*}
&&\norm u-P_hu\norm_{\Omega_h\setminus\partial\Omega_h}\le a\norm \D\cdot u\norm_{\Omega_h^\circ}+b\max_{x\in \tilde{B}_h}|u(x)|,\\
&&\tilde{B}_h:=(\Omega_h\setminus\Omega_h^{\circ})\cup \partial\Omega_h^{\circ}\mbox{\rm \quad(see Subsection 2.1 for the definition of $\Omega_h^\circ$)}.
\end{eqnarray*}
\end{Lemma}
\begin{proof}
There exists $\phi:\Omega_h\to\R$ with the zero mean such that $u=P_hu+\D \phi$ on $\Omega_h\setminus\partial\Omega_h$. We will apply the Poincar\'e type inequality II to $\phi$ (that is why $\Omega_h^\circ$ is involved).  With discrete divergence free constraint of $P_hu$, we have   
\begin{eqnarray}\label{divdivdiv}
&&\norm u-P_hu\norm_{\Omega_h\setminus\partial\Omega_h}^2=\norm u-P_hu\norm_{\Omega_h\setminus\partial\Omega_h}\norm \D\phi\norm_{\Omega_h\setminus\partial\Omega_h}\\ \nonumber
&&\qquad=\sum_{x\in \Omega_h\setminus\partial\Omega_h} (u(x)-P_hu(x))\cdot \D\phi(x)h^3\\ \nonumber
&&\qquad=\sum_{x\in \Omega_h\setminus\partial\Omega_h} u(x)\cdot \D\phi(x)h^3\\\nonumber 
&&\qquad=\sum_{x\in (\Omega_h\setminus\partial\Omega_h)\setminus\Omega_h^{\circ}} u(x)\cdot \D\phi(x)h^3+\sum_{x\in \Omega_h^{\circ}} u(x)\cdot \D\phi(x)h^3\\\nonumber
&&\qquad\le\max_{\tilde{B}_h}|u(x)|  \norm \D\phi\norm_{\Omega_h\setminus\partial\Omega_h}+ \sum_{x\in \Omega_h^{\circ}} u(x)\cdot \D\phi(x)h^3.
\end{eqnarray}
Set 
\begin{eqnarray*}
&&\tilde{\Gamma}_h^{\circ j+}:= \{x\in\partial\Omega_h^\circ\,|\,x+he^j\not\in\Omega_h^\circ,\,\,\, x-he^j\in\Omega_h^\circ  \},\\ 
&&\tilde{\Gamma}_h^{\circ j-}:= \{x\in\partial\Omega_h^\circ\,|\,x-he^j\not\in\Omega_h^\circ,\,\,\,x+he^j\in\Omega_h^\circ\},\\
&&\tilde{\Gamma}_h^{\circ j}:= \{x\in\partial\Omega_h^\circ\,|\,x\pm he^j\not\in\Omega_h^\circ\}.   
\end{eqnarray*}
With  the Poincar\'e type inequality II, we have  
\begin{eqnarray}\label{divdivdiv2}
&&\sum_{x\in \Omega_h^\circ} u(x)\cdot \D\phi(x)h^3
= -\sum_{x\in \Omega_h^\circ} \D\cdot u(x)\phi(x)h^3\\\nonumber 
&&\qquad + \sum_{j=1}^3\Big\{\sum_{x\in \tilde{\Gamma}_h^{\circ j+}}u_j(x)\phi(x+he^j) -\sum_{x\in \tilde{\Gamma}_h^{\circ j-}}u_j(x-he^j)\phi(x) \\ \nonumber
&&\qquad\qquad  -\sum_{x\in \tilde{\Gamma}_h^{\circ j-}}u_j(x)\phi(x-he^j)
+\sum_{x\in \tilde{\Gamma}_h^{\circ j+}}u_j(x+he^j)\phi(x) \Big\}\frac{h^3}{2h}  \\\nonumber
&&\qquad  +\sum_{j=1}^3\sum_{x\in \tilde{\Gamma}_h^{\circ j}}u_j(x) D_j\phi(x)h^3\\\nonumber
&&\le  \norm \D\cdot u\norm_{\Omega_h^\circ}\norm \phi \norm_{\Omega_h^\circ}
+\max_{\tilde{B}_h}|u(x)|\sum_{j=1}^3\Big\{
 \sum_{x\in \tilde{\Gamma}_h^{\circ j+}} (|\phi(x+he^j)| +|\phi(x)|)        \\ \nonumber
&&\qquad +\sum_{x\in \tilde{\Gamma}_h^{\circ j-}}(|\phi(x-he^j)|+|\phi(x)|)
\Big\}\frac{h^3}{2h} + 3 \max_{\tilde{B}_h}|u(x)|  \beta h^\2\norm \D\phi\norm_{\Omega_h\setminus\partial\Omega_h}    \\    \nonumber
&&\le  \tilde{A}\norm \D\cdot u\norm_{\Omega_h^\circ }\norm \D\phi \norm_{\Omega_h\setminus\partial\Omega_h} +3 \max_{\tilde{B}_h}|u(x)|  \beta h^\2\norm \D\phi\norm_{\Omega_h\setminus\partial\Omega_h}+r,\end{eqnarray}
\begin{eqnarray*}
\quad r=\max_{\tilde{B}_h}|u(x)|\sum_{j=1}^3\Big\{
 \sum_{x\in \tilde{\Gamma}_h^{\circ j+}} (|\phi(x+he^j)| +|\phi(x)|)      
+\sum_{x\in \tilde{\Gamma}_h^{\circ j-}}(|\phi(x-he^j)|+|\phi(x)|)
\Big\}\frac{h^3}{2h}.  
\end{eqnarray*}
We estimate the terms in $\{\,\,\,\}$ of $r$, where \eqref{jiro1}, \eqref{jiro2} are not available because $\phi$ is not estimated in $L^\infty$ and  \eqref{jiro2} leaves $O(h^{-\2})$.  Take a smooth function $\gamma(s):[0,s_0]\to[0,1]$ such that 
\begin{eqnarray*}
\gamma(s)&=&\left\{
\begin{array}{lll}
&1,\qquad s\in[0,\frac{s_0}{3}],
\medskip\\
&0, \qquad s\in[\frac{2s_0}{3},s_0],
\end{array}
\right.
\end{eqnarray*}    
where $s_0$ is the one in Condition A. Now we use Condition A. Since $V_k$ ($k=1,\cdots,K$) are open subsets of planes, we still have the statements of Condition A with $\{V_k-\ep\}_{k=1,\ldots, K}$ instead of $\{V_k\}_{k=1,\ldots, K}$ for some $\ep>0$, where $V_k-\ep$ stands for the set $V_k\setminus\{ x\in\R^3\,|\,|x-y|<\ep, y\in\partial V_k \}$. Define 
$$B_k^\ep:=\{ y+s \omega_k\,|\, y\in (V_k-\ep),\,\,\,s\in[0,\varphi_k(y)] \}.$$
For $h\ll\ep$, we have the following estimate: Fix $l^\ast\in\N$ such that $2l^\ast h\in [\frac{2s_0}{3},s_0]$;  For each $x\in\tilde{\Gamma}_h^{\circ j+}\cap B_k^\ep$, we have 
\begin{eqnarray*}
\phi(x)&=&\phi(x)\gamma(0),\\
&=&\phi(x)\gamma(0)-\phi(x+2h\omega_k)\gamma(2h)
+\phi(x+2h\omega_k)\gamma(2h)-\phi(x+4h\omega_k )\gamma(4h)
+\cdots\\
&&+\phi(x+2(l^\ast-1)h\omega_k )\gamma(2(l^\ast-1)h)-\phi(x+2l^\ast h\omega_k )\gamma(2h)\\
&=&\sum_{l=1}^{l^\ast}\Big( \frac{\phi(x+2 (l-1)h\omega_k)-\phi(x+2 lh\omega_k)}{2h}\gamma(2(l-1)h)\\
&&+  \phi(x+2lh\omega_k)\frac{\gamma(2(l-1)h)-\gamma(2lh)}{2h}  \Big)\cdot2h,\\
|\phi(x)|&\le& \sum_{l=1}^{l^\ast} \Big(|\D\phi(x+(2l-1)h\omega_k)|+\beta_1|\phi(x+2lh\omega_k)|\Big)\cdot2h,
\end{eqnarray*}
where $\beta_1=\max|\gamma'|$. Note that $x+2lh\omega_k\in\Omega_h^\circ$ for all $l=1,\ldots,l^\ast$ because of $V_k-\varepsilon$ instead of $V_k$.  Hence, we see that  
\begin{eqnarray*}
\sum_{x\in \tilde{\Gamma}_h^{\circ j+}\cap B_k^\ep } |\phi(x)| \frac{h^3}{2h}&\le& 
\sum_{x\in\Omega_h\setminus\partial\Omega_h} |\D\phi(x)|h^3+ \sum_{x\in\Omega_h^{\circ}}\beta_1|\phi(x)|h^3\\
&\le& 
 \beta_2\norm \D\phi\norm_{\Omega_h\setminus\partial\Omega_h}+\beta_3 \norm \phi\norm_{\Omega_h^\circ}\\
 &\le&  \beta_2\norm \D\phi\norm_{\Omega_h\setminus\partial\Omega_h}+\beta_3 \tilde{A}\norm \D\phi\norm_{\Omega_h\setminus\partial\Omega_h},\\
\sum_{x\in \tilde{\Gamma}_h^{\circ j+}} |\phi(x)|\frac{h^3}{2h} 
& \le&  K(\beta_2+\beta_3\tilde{A})\norm \D\phi\norm_{\Omega_h\setminus\partial\Omega_h}.
\end{eqnarray*}
For $x\in\tilde{\Gamma}_h^{\circ j+}\cap B_k^\ep$, we have 
\begin{eqnarray*}
&&\sum_{x\in\tilde{\Gamma}_h^{\circ j+}}|\phi(x+he^j)|\frac{h^3}{2h}=\sum_{x\in\tilde{\Gamma}_h^{\circ j+}}|D_j\phi(x)|h^3
+\sum_{x\in\tilde{\Gamma}_h^{\circ j+}}|\phi(x-he^j)|\frac{h^3}{2h}\\
&&\quad\le \beta h^\2\norm \D\phi\norm_{\Omega_h\setminus\partial\Omega_h}+\sum_{x\in\tilde{\Gamma}_h^{\circ j+}}|\phi(x-he^j)|\frac{h^3}{2h},\quad x-he^j\in\Omega_h^\circ.
\end{eqnarray*}
We have the same estimate for $|\phi(x-he^j)|$.
In this way, we obtain 
\begin{eqnarray}\label{divdivdiv3}
r\le \beta_4\max_{\tilde{B}_h}|u(x)|\norm \D\phi\norm_{\Omega_h\setminus\partial\Omega_h}
\end{eqnarray}
with some constant $\beta_4>0$. \eqref{divdivdiv}, \eqref{divdivdiv2} and \eqref{divdivdiv3} conclude the proof. 
\end{proof}  
\begin{Lemma}\label{div-smooth}
For each function $u:\Omega\to\R^3$ such that $u\in C^2(\bar{\Omega})$, $u|_{\partial\Omega}=0$ and $\nabla\cdot u=0$, there exists a constant $\beta>0$ independent of $h$ for which we have 
$$\norm u-P_hu\norm_{\Omega_h\setminus\partial\Omega_h} \le \beta h,$$
where $P_hu=P_h(u|_{\Omega_h})$ with the $0$-extension of $u|_{\Omega_h}$ outside $\Omega_h$.  
\end{Lemma}
\begin{proof}
For each $x\in\Omega_h\setminus\partial\Omega_h$, we have 
$$\D\cdot u(x)=\sum_{j=1}^3\frac{u_j(x+he^j)-u_j(x-he^j)}{2h}=\nabla\cdot u(x)+O(h)=O(h).$$
It follows from \eqref{distance1}  that, for each $x\in \tilde{B}_h$, we have $x^\ast\in\partial\Omega$ such that 
$$|u(x)|=|u(x)-u(x^\ast)|=O(h).$$ 
Lemma \ref{lem-divdivdiv} yields the assertion. 
\end{proof}
\begin{Lemma}\label{B.V.}
There exists a constant $\beta>0$ depending only on $\Omega$ such that for any function $\phi:\Omega_h\to\R$ and $a=\pm1,\pm2$ we have  
\begin{eqnarray*}
&&\sum_{j=1}^3\Big(\sum_{x\in\Gamma_h^{j\pm}}|\phi(x+ahe^j)|^2 h^3\Big)^\2 \le \beta h^\2\Big( \sum_{j=1}^3\norm D_j^+ \phi \norm_{\Omega_h}+\norm  \phi \norm_{\Omega_h\setminus\partial\Omega_h}\Big).
\end{eqnarray*}
\end{Lemma}
\begin{proof}
We deal with the case of $a=-2$. Due to the same reasoning and notation as those of the proof of Lemma \ref{lem-divdivdiv}, we have for each $x\in \Gamma_h^{j+}\cap B_k^\ep$,
\begin{eqnarray*}
 |\phi(x-2he^j)| &\le& 
\sum_{l=0}^{2l^\ast-1} \Big(\Big|\frac{\phi(x-2he^j+(l+1)h\omega_k )-\phi(x-2he^j+lh\omega_k )}{h}\Big|\\
&&+\beta_1|\phi(x-2he^j+(l+1)h\omega_k)|\Big)h\\
&\le& \beta_2\Big(\sum_{l=0}^{2l^\ast-1} \Big|\frac{\phi(x-2he^j+(l+1)h\omega_k )-\phi(x-2he^j+lh\omega_k )}{h}\Big|^2h\\
&&+ \sum_{l=0}^{2l^\ast-1} |\phi(x-2he^j+(l+1)h\omega_k)|^2h\Big)^\2,\\
 |\phi(x-2he^j)|^2 &\le& 
\beta_2^2\sum_{l=0}^{2l^\ast-1} \Big|\frac{\phi(x-2he^j+(l+1)h\omega_k )-\phi(x-2he^j+lh\omega_k )}{h}\Big|^2h\\
&&+\beta_2^2 \sum_{l=0}^{2l^\ast-1} |\phi(x-2he^j+(l+1)h\omega_k)|^2h.
\end{eqnarray*}
Hence we see that 
\begin{eqnarray*}
\sum_{x\in\Gamma_h^{j+}\cap B_k^\ep}|\phi(x-2he^j)|^2 h^3
&\le& \beta_2^2h\sum_{j=1}^3\norm D^+_j\phi\norm_{\Omega_h}^2+\beta_2^2h\norm\phi\norm_{\Omegain}^2,\\
\sum_{x\in\Gamma_h^{j+}}|\phi(x-2he^j)|^2 h^3
&\le& \beta_2^2hK\sum_{j=1}^3\norm D^+_j\phi\norm_{\Omega_h}^2+\beta_2^2hK\norm\phi\norm_{\Omegain}^2,\\
\Big(\sum_{x\in\Gamma_h^{j+}}|\phi(x-2he^j)|^2 h^3\Big)^\2
&\le& \beta_3 h^\2\Big(\sum_{j=1}^3\norm D^+_j\phi\norm_{\Omega_h}+\norm\phi\norm_{\Omegain}\Big).
\end{eqnarray*}
The other cases are proved in the same way.
\end{proof}
\subsection{Proof of Theorem \ref{error-estimate}}
For each $n\ge0$, we have with Lemma \ref{div-smooth},
\begin{eqnarray}\label{v}
\norm b^n\norm_{\Omegain}
&\le&\norm P_h\tilde{b}^n\norm_{\Omegain}
+\norm v^n-P_hv^n\norm_{\Omegain}\le \norm \tilde{b}^n\norm_{\Omegain}
+\beta h.
\end{eqnarray}
We take the inner product of \eqref{bbbbb} and $\tilde{b}^{n+1}$ over $\Omega_h\setminus\partial\Omega_h$: 
Observe that  
\begin{eqnarray*}
-\sum_{x\in\Omega_h\setminus\partial\Omega_h}({\rm i})\cdot\tilde{b}^{n+1}(x)h^3&=&\frac{\tau}{2}\sum_{x\in\Omega_h\setminus\partial\Omega_h}\D\cdot u^n(x)|\tilde{b}^{n+1}(x)|^2h^3\\
&&+\frac{\tau}{2}\sum_{j=1}^3\sum_{x\in\Omega_h\setminus\partial\Omega_h}
\frac{1}{2h}u^n_j(x-he^j)\tilde{b}^{n+1}(x-2he^j)\cdot\tilde{b}^{n+1}(x)h^3\\
&&-\frac{\tau}{2}\sum_{j=1}^3\sum_{x\in\Omega_h\setminus\partial\Omega_h}
\frac{1}{2h}u^n_j(x+he^j)\tilde{b}^{n+1}(x+2he^j)\cdot\tilde{b}^{n+1}(x)h^3\\
&=&\frac{\tau}{2}\sum_{x\in\Omega_h\setminus\partial\Omega_h}\D\cdot u^n(x)|\tilde{b}^{n+1}(x)|^2h^3\\
&&-\frac{\tau}{2}\sum_{j=1}^3\sum_{x\in\Gamma_h^{j+}}
\frac{1}{2h}u^n_j(x-he^j)\tilde{b}^{n+1}(x-2he^j)\cdot\tilde{b}^{n+1}(x)h^3\\
&&+\frac{\tau}{2}\sum_{j=1}^3\sum_{x\in\Gamma_h^{j-}}
\frac{1}{2h}u^n_j(x+he^j)\tilde{b}^{n+1}(x+2he^j)\cdot\tilde{b}^{n+1}(x)h^3.
\end{eqnarray*}
With the discrete divergence free constraint of $u^n$ and $\tilde{b}^{n+1}(x)=-v^{n+1}(x)=O(h)$ for $x\in\Gamma_h^{j\pm}$, we obtain 
\begin{eqnarray*}
&&
-\sum_{x\in\Omega_h\setminus\partial\Omega_h}({\rm i})\cdot\tilde{b}^{n+1}(x)h^3\le
O(\tau)\sum_{j=1}^3\sum_{x\in\Gamma_h^{j+}}
|u^n_j(x-he^j)\tilde{b}^{n+1}(x-2he^j)|h^3\\
&&\qquad+O(\tau)\sum_{j=1}^3\sum_{x\in\Gamma_h^{j-}}
|u^n_j(x+he^j)\tilde{b}^{n+1}(x+2he^j)|h^3\\
&&\quad \le O(\tau)\sum_{j=1}^3\sum_{x\in\Gamma_h^{j+}}
|u^n_j(x-he^j)-v^{n}_j(x-he^j)||\tilde{b}^{n+1}(x-2he^j)|h^3  \\
&&\qquad+O(\tau)\sum_{j=1}^3\sum_{x\in\Gamma_h^{j+}}
|v_j^{n}(x-he^j)||\tilde{b}^{n+1}(x-2he^j)|h^3  \\
&&\qquad+O(\tau)\sum_{j=1}^3\sum_{x\in\Gamma_h^{j-}}
|u^n_j(x+he^j)-v^n_j(x+he^j)||\tilde{b}^{n+1}(x+2he^j)|h^3\\
&&\qquad+O(\tau)\sum_{j=1}^3\sum_{x\in\Gamma_h^{j-}}
|v^n_j(x+he^j)||\tilde{b}^{n+1}(x+2he^j)|h^3\\
&&\quad= O(\tau)\sum_{j=1}^3\sum_{x\in\Gamma_h^{j+}}
|b^{n}_j(x-he^j)||\tilde{b}^{n+1}(x-2he^j)|h^3  \\
&&\qquad+O(\tau)\sum_{j=1}^3\sum_{x\in\Gamma_h^{j-}}
|b^n_j(x+he^j)||\tilde{b}^{n+1}(x+2he^j)|h^3\\
&&\qquad+O(\tau h)\sum_{j=1}^3\sum_{x\in\Gamma_h^{j+}}
|\tilde{b}^{n+1}(x-2he^j)|h^3  \\
&&\qquad+O(\tau h)\sum_{j=1}^3\sum_{x\in\Gamma_h^{j-}}
|\tilde{b}^{n+1}(x+2he^j)|h^3. 
\end{eqnarray*}
We estimate each term: By Lemma \ref{B.V.} and \eqref{v}, we have
\begin{eqnarray*}
&&\sum_{j=1}^3\sum_{x\in\Gamma_h^{j+}}
|b^{n}_j(x-he^j)||\tilde{b}^{n+1}(x-2he^j)|h^3\\
&&\quad \le   \sum_{j=1}^3\Big(\sum_{x\in\Gamma_h^{j+}}|b^{n}_j(x-he^j)|^2h^3\Big)^\2  \Big(\sum_{x\in\Gamma_h^{j+}} |\tilde{b}^{n+1}(x-2he^j)|^2h^3\Big)^\2  \\
&&\quad \le  O(h^\2) \norm  b^n \norm_{\Omega_h\setminus\partial\Omega_h}  \Big(\sum_{j=1}^3 \norm D^+_j \tilde{b}^{n+1} \norm_{\Omega_h} + \norm\tilde{b}^{n+1}\norm_{\Omegain} \Big)\\
&&\quad \le O(h^\2) (\norm  \tilde{b}^n \norm_{\Omega_h\setminus\partial\Omega_h}  +\beta h)\Big(\sum_{j=1}^3 \norm D^+_j \tilde{b}^{n+1} \norm_{\Omega_h} + \norm\tilde{b}^{n+1}\norm_{\Omegain} \Big);
\end{eqnarray*}
\begin{eqnarray*}
&&\sum_{j=1}^3\sum_{x\in\Gamma_h^{j+}}
|\tilde{b}^{n+1}(x-2he^j)|h^3\le 
\sum_{j=1}^3\Big(\sum_{x\in\Gamma_h^{j+}}
h^3\Big)^\2\Big(\sum_{x\in\Gamma_h^{j+}}
|\tilde{b}^{n+1}(x-2he^j)|^2h^3\Big)^\2\\
&&\quad \le O(h^\2)O(h^\2)\Big(\sum_{j=1}^3\norm D^+_j \tilde{b}^{n+1} \norm_{\Omega_h} + \norm\tilde{b}^{n+1}\norm_{\Omegain} \Big); 
\end{eqnarray*}
The other terms are also estimated in this way. Hence, we obtain 
\begin{eqnarray}\label{(i)}
-\sum_{x\in\Omega_h\setminus\partial\Omega_h}({\rm i})\cdot\tilde{b}^{n+1}(x)h^3&\le&O(\tau h^2)\norm \tilde{b}^{n+1}\norm_{\Omega_h\setminus\partial\Omega_h} \\\nonumber
&&+O(\tau h^\2)\norm \tilde{b}^n\norm_{\Omegain}\norm \tilde{b}^{n+1}\norm_{\Omegain} \\\nonumber 
&&+O(\tau h^\2)\norm \tilde{b}^{n}\norm_{\Omega_h\setminus\partial\Omega_h}   \sum_{j=1}^3 \norm D^+_j \tilde{b}^{n+1} \norm_{\Omega_h}\\\nonumber
&&+O(\tau h^\frac{3}{2}) \sum_{j=1}^3 \norm D^+_j \tilde{b}^{n+1} \norm_{\Omega_h}.
\end{eqnarray}
\indent Since $b^n=-v^n=O(h)$ outside $\Omegain$, we have with \eqref{v},
\begin{eqnarray}\label{(ii)}
&&-\sum_{x\in\Omegain}({\rm ii})\cdot \tilde{b}^{n+1}(x)h^3
\le O(\tau)\norm b^n \norm_{\Omega_h}\norm \tilde{b}^{n+1}\norm_{\Omegain}\\\nonumber
&&\qquad \le O(\tau)\Big(\norm b^n \norm_{\Omegain}^2+O(h^3)\Big)^\2\norm \tilde{b}^{n+1}\norm_{\Omegain}\\\nonumber
&&\qquad \le O(\tau)\Big(\norm \tilde{b}^n\norm_{\Omegain}+\beta h +O(h^\frac{3}{2})\Big)\norm \tilde{b}^{n+1}\norm_{\Omegain}\\\nonumber
&&\qquad \le O(\tau)\norm \tilde{b}^n \norm_{\Omegain}\norm \tilde{b}^{n+1}\norm_{\Omegain}+O(\tau h)\norm \tilde{b}^{n+1}\norm_{\Omegain},
\end{eqnarray}
where we took $\max |D_jv^{n+1}|$ out of the inner product. 

Observe that 
\begin{eqnarray*}
&&\sum_{x\in\Omegain} ({\rm iii})\cdot\tilde{b}^{n+1}(x)h^3\\
&&\qquad= \tau\sum_{j=1}^3\sum_{x\in \Omega_h} D_j^2\tilde{b}^{n+1}(x)\cdot\tilde{b}^{n+1}(x)h^3 -\tau\sum_{j=1}^3\sum_{x\in\partial\Omega_h} D_j^2\tilde{b}^{n+1}(x)\cdot\tilde{b}^{n+1}(x)h^3  \\
&&\qquad=-\tau\sum_{j=1}^3\norm D_j^+\tilde{b}^{n+1}\norm_{\Omega_h}^2 +r,\\
&&r=-\tau\sum_{j=1}^3\sum_{x\in\partial\Omega_h} D_j^2\tilde{b}^{n+1}(x)\cdot\tilde{b}^{n+1}(x)h^3 
+\tau\sum_{j=1}^3\sum_{x\in \tilde{\Gamma}_h^{j+}} \frac{1}{h}D_j^+\tilde{b}^{n+1}(x)\cdot\tilde{b}^{n+1}(x+he^j) h^3\\
&&\,\,\,\,\,\quad -\tau\sum_{j=1}^3\sum_{x\in \tilde{\Gamma}_h^{j-}} \frac{1}{h}D_j^+\tilde{b}^{n+1}(x-he^j)\cdot\tilde{b}^{n+1}(x) h^3.
\end{eqnarray*}
Since $\tilde{b}^{n+1}=-v^{n+1}=O(h)$ outside $\Omegain$, we have 
\begin{eqnarray*}
r&=&\tau\sum_{j=1}^3\sum_{x\in\partial\Omega_h} \frac{D_j^+\tilde{b}^{n+1}(x)-D_j^+\tilde{b}^{n+1}(x-he^j)}{h}\cdot v^{n+1}(x)h^3 \\
&&+\tau\sum_{j=1}^3\sum_{x\in \tilde{\Gamma}_h^{j+}} \frac{1}{h}\frac{v^{n+1}(x+he^j)-v^{n+1}(x)}{h}\cdot v^{n+1}(x+he^j) h^3\\
&&-\tau\sum_{j=1}^3\sum_{x\in \tilde{\Gamma}_h^{j-}} \frac{1}{h}\frac{v^{n+1}(x)-v^{n+1}(x-he^j)}{h}\cdot v^{n+1}(x) h^3\\
&\le&O(\tau)\sum_{j=1}^3\sum_{x\in\partial\Omega_h}( |D_j^+\tilde{b}^{n+1}(x)|+|D_j^+\tilde{b}^{n+1}(x-he^j)|)h^3 
+O(\tau h)\\
&\le&O(\tau)\sum_{j=1}^3\sum_{x\in\Omega_h} |D_j^+\tilde{b}^{n+1}(x)|\chi_{\partial \Omega_h}(x)h^3\\
&&+O(\tau)\sum_{j=1}^3\sum_{x\in\Omega_h-he^j}|D_j^+\tilde{b}^{n+1}(x)|\chi_{\partial\Omega_h-he^j}(x) h^3 
+O(\tau h)\\
&\le&O(\tau h^\2)\sum_{j=1}^3\norm D_j^+\tilde{b}^{n+1}\norm_{\Omega_h}
+O(\tau h^\2)\sum_{j=1}^3\norm D_j^+\tilde{b}^{n+1}\norm_{\Omega_h-he^j} +O(\tau h) \\
&\le&O(\tau h^\2)\sum_{j=1}^3\norm D_j^+\tilde{b}^{n+1}\norm_{\Omega_h}+O(\tau h),
\end{eqnarray*}
where we note that $\Omega_h-he^j:=\{x\,|\,x-he^j,\,\,x\in\Omega_h\}$, $\partial\Omega_h-he^j:=\{x\,|\,x-he^j,\,\,x\in\partial\Omega_h\}$ and 
\begin{eqnarray}\label{hosoku}
&&\norm D_j^+\tilde{b}^{n+1}\norm_{\Omega_h-he^j}=
\Big(\sum_{x\in(\Omega_h-he^j)\cap\Omega_h}\!\!\!\!|D_j^+\tilde{b}^{n+1}(x)|^2h^3
+\sum_{x\in(\Omega_h-he^j)\setminus\Omega_h}\!\!\!\!|D_j^+\tilde{b}^{n+1}(x)|^2h^3\Big)^\2\\\nonumber
&&\qquad \le\Big(\norm D_j^+\tilde{b}^{n+1}\norm_{\Omega_h}^2
+\sum_{x\in(\Omega_h-he^j)\setminus\Omega_h}O(1)h^3\Big)^\2\le\Big(\norm D_j^+\tilde{b}^{n+1}\norm_{\Omega_h}^2
+O(h)\Big)^\2\\\nonumber
&&\qquad \le \norm D_j^+\tilde{b}^{n+1}\norm_{\Omega_h}+O(h^\2).
\end{eqnarray}
By the discrete Poincar\'e type inequality I, where we note that $\tilde{b}^{n+1}|_{\partial\Omega_h}\neq0$ and the discrete Poincar\'e type inequality I does not work for $\tilde{b}^{n+1}$ itself,  we have 
\begin{eqnarray}\label{Poin-poin}
&&\norm D^+_j\tilde{b}^{n+1}\norm_{\Omega_h}\ge 
\norm D_j^+(\tilde{b}^{n+1}-\tilde{b}^{n+1}\chi_{\partial\Omega_h} ) \norm_{\Omega_h}
-\norm D_j^+(\tilde{b}^{n+1}\chi_{\partial\Omega_h} ) \norm_{\Omega_h}\\ \nonumber
&&\quad \ge A^{-1} \norm \tilde{b}^{n+1}-\tilde{b}^{n+1}\chi_{\partial\Omega_h}\norm_{\Omegain}-\frac{2}{h}\Big(\sum_{x\in\Omega_h} |\tilde{b}^{n+1}(x)|^2\chi_{\partial\Omega_h}(x)^2 h^3\Big)^\2\\\nonumber
&&\quad=A^{-1}\norm  \tilde{b}^{n+1}\norm_{\Omegain}
-\frac{2}{h}\Big(\sum_{x\in\partial\Omega_h} |v^{n+1}(x)|^2 h^3 \Big)^\2\\\nonumber
&&\quad=A^{-1}\norm  \tilde{b}^{n+1}\norm_{\Omegain}-O(h^\2).
\end{eqnarray}
Hence, we obtain 
\begin{eqnarray}\label{(iii)}
&&\sum_{x\in\Omega_h\setminus\partial\Omega_h}({\rm iii})\cdot\tilde{b}^{n+1}(x)h^3\le-\tau A^{-1}\norm \tilde{b}^{n+1}\norm_{\Omegain} \sum_{j=1}^3\norm D_j^+\tilde{b}^{n+1}\norm_{\Omega_h}\\\nonumber
&&\quad +O(\tau h^\2)\sum_{j=1}^3\norm D_j^+\tilde{b}^{n+1}\norm_{\Omega_h}+O(\tau h).
\end{eqnarray}
Observe that 
\begin{eqnarray*}
\sum_{x\in\Omegain} -R^n(x)\cdot\tilde{b}^{n+1}(x)h^3&\le& 
\tau\sum_{x\in\Omegain} \D p^n(x)\cdot \tilde{b}^{n+1}(x)h^3\\
&& +(O(\tau h)+O(\tau^2))\norm \tilde{b}^{n+1}\norm_{\Omegain}. 
\end{eqnarray*}
With \eqref{bbbbb}, we have 
\begin{eqnarray*}
&&\tau\sum_{x\in\Omegain} \D p^n(x)\cdot \tilde{b}^{n+1}(x)h^3
 =\tau\sum_{x\in\Omegain} \D p^n(x)\cdot \Big\{
b^n(x)\\
&&\qquad\qquad- \frac{\tau}{2}\sum_{j=1}^3   
\Big(u^n_j(x-he^j)D^+_j\tilde{b}^{n+1}(x-he^j)+u^n_j(x+he^j)D^+_j\tilde{b}^{n+1}(x+he^j)\Big)\\
&&\qquad\qquad-\frac{\tau}{2}\sum_{j=1}^3 \Big(b^n_j(x-he^j)D^+_jv^{n+1}(x-he^j)+b^n_j(x+he^j)D^+_jv^{n+1}(x+he^j) \Big)\\
&&\qquad\qquad +\tau\sum_{j=1}^3D^2_j\tilde{b}^{n+1}(x) +\tau\nabla p^n(x)+O(\tau h)+O(\tau^2)  \Big\}h^3. 
\end{eqnarray*}
We estimate each term: 
Since $\D\cdot b^n(x)=-\D\cdot v^n(x)=O(h^2)$ on $\Omega_h$ and $b^n=-v^n=O(h)$ on $\partial\Omega_h$, we have 
\begin{eqnarray*}
&&\!\!\!\!\!\!\tau\sum_{x\in\Omegain} \D p^n(x)\cdot b^n(x) h^3
=\tau\sum_{x\in\Omega_h} \D p^n(x)\cdot b^n(x)h^3-\tau\sum_{x\in\partial\Omega_h} \D p^n(x)\cdot b^n(x)h^3\\
&& \quad= -\tau\sum_{j=1}^3\sum_{x\in\Omega_h} p^n(x)\big(\D\cdot  b^n (x)\big)h^3 \\
&&\qquad +\tau\sum_{j=1}^3 \Big\{ \sum_{x\in\tilde{\Gamma}_h^{j+}} p^n(x+he^j)b_j^n(x) 
 -\sum_{x\in\tilde{\Gamma}_h^{j-}} p^n(x)b_j^n(x-he^j)  \\
&&\qquad
 -\sum_{x\in\tilde{\Gamma}_h^{j-}} p^n(x)b_j^n(x-he^j)
+\sum_{x\in\tilde{\Gamma}_h^{j+}} p^n(x)b_j^n(x+he^j)        \Big\}\frac{h^3}{2h}+O(\tau h^2)\\
&&\quad= O(\tau h);\\\\
&&\!\!\!\!\!\!-\frac{\tau^2}{2}\sum_{x\in\Omegain} \D p^n(x)\cdot \Big\{
\sum_{j=1}^3   
\Big(u^n_j(x-he^j)D^+_j\tilde{b}^{n+1}(x-he^j) \\
&&\qquad +u^n_j(x+he^j)D^+_j\tilde{b}^{n+1}(x+he^j)\Big)\Big\}h^3\\
&&\quad \le O(\tau^2)\norm u^n\norm_{\Omegain}\sum_{j=1}^3\norm D_j^+\tilde{b}^{n+1} \norm_{\Omega_h}
= O(\tau^2)\sum_{j=1}^3\norm D_j^+\tilde{b}^{n+1} \norm_{\Omega_h};\\\\
&&\!\!\!\!\!\!-\frac{\tau^2}{2}\sum_{x\in\Omegain} \D p^n(x)\cdot \Big\{
\sum_{j=1}^3   
 \Big(b^n_j(x-he^j)D^+_jv^{n+1}(x-he^j)\\
 &&\qquad +b^n_j(x+he^j)D^+_jv^{n+1}(x+he^j) \Big)\Big\}h^3\\
&&\quad \le O(\tau^2)\norm b^n\norm_{\Omega_h}
\le O(\tau^2)\norm b^n\norm_{\Omegain}+O(\tau^2h^{\frac{3}{2}}) \\
&&\quad \le O(\tau^2)(\norm \tilde{b}^n\norm_{\Omegain}+\beta h)+O(\tau^2h^{\frac{3}{2}})
 \le O(\tau^2)\norm \tilde{b}^n\norm_{\Omegain}+O(\tau^2h) ;
\end{eqnarray*}
Since $p\in C^2$ and $\tilde{b}^{n+1}=-v^{n+1}$ on $\partial\Omega_h$, we have with \eqref{hosoku};
\begin{eqnarray*}
&&\tau^2\sum_{x\in\Omegain} \D p^n(x)\cdot \Big\{
\sum_{j=1}^3D^2_j\tilde{b}^{n+1}(x)  \Big\}h^3\\
&&\quad \le \tau^2\sum_{j=1}^3\sum_{x\in\Omega_h} \D p^n(x)\cdot D^2_j\tilde{b}^{n+1}(x) h^3
-\tau^2\sum_{j=1}^3\sum_{x\in\partial\Omega_h} \D p^n(x)\cdot D^2_j\tilde{b}^{n+1}(x) h^3\\
&&\quad\le  -\tau^2\sum_{j=1}^3\sum_{x\in\Omega_h} D_j^+(\D p^n)(x)\cdot D^+_j\tilde{b}^{n+1}(x) h^3\\
&&\qquad + \tau^2\sum_{j=1}^3\Big(\sum_{x\in\tilde{\Gamma}_h^{j+}}\frac{1}{h}\D p^n(x+he^j)\cdot D^+_j\tilde{b}^{n+1}(x) 
- \sum_{x\in\tilde{\Gamma}_h^{j-}}\frac{1}{h}\D p^n(x)\cdot D^+_j\tilde{b}^{n+1}(x-he^j)\Big)h^3\\
&&\qquad-\tau^2\sum_{j=1}^3\sum_{x\in\partial\Omega_h} \D p^n(x)\cdot \frac{D^+_j\tilde{b}^{n+1}(x)-D^+_j\tilde{b}^{n+1}(x-he^j)}{h} h^3\\
&&\quad \le O(\tau^2)\sum_{j=1}^3 \norm D_j^+\tilde{b}^{n+1} \norm_{\Omega_h}+O(\tau^2)+O(\frac{\tau^2}{h})\sum_{j=1}^3\sum_{x\in\Omega_h}\chi_{\partial\Omega_h}(x)|D_j^+\tilde{b}^{n+1}(x)|h^3 \\
&&\qquad +O(\frac{\tau^2}{h})\sum_{j=1}^3\sum_{x\in\Omega_h-he^j}\chi_{\partial\Omega_h-he^j}(x)|D^j_+\tilde{b}^{n+1}(x)| h^3\\
&&\quad \le  O(\frac{\tau^2}{h^\2})\sum_{j=1}^3 \norm D_j^+\tilde{b}^{n+1} \norm_{\Omega_h}+O(\tau^2);\\\\
&&\tau\sum_{x\in\Omegain} \D p^n(x)\cdot\Big(\tau\nabla p^n(x)+O(\tau h)+O(\tau^2)  \Big)h^3\le O(\tau^2).
\end{eqnarray*}
Therefore, we obtain 
\begin{eqnarray}\label{(iiii)}
&&-\sum_{x\in\Omegain} R^n(x)\cdot\tilde{b}^{n+1}(x)h^3
\le   (O(\tau h)+O(\tau^2))\norm \tilde{b}^{n+1}\norm_{\Omegain}\\\nonumber
 &&\qquad+O(\tau^2)\norm \tilde{b}^n\norm_{\Omegain}+ O(\frac{\tau^2}{h^\2})\sum_{j=1}^3 \norm D_j^+\tilde{b}^{n+1} \norm_{\Omega_h}
+ O(\tau h)+O(\tau^2).
\end{eqnarray}
Finally, we have 
\begin{eqnarray}\label{order2}
\sum_{x\in\Omegain}b^n(x)\cdot\tilde{b}^{n+1}(x)h^3\le\norm \tilde{b}^n\norm_\Omegain\norm \tilde{b}^{n+1}\norm_\Omegain+\beta h\norm \tilde{b}^{n+1}\norm_\Omegain.
\end{eqnarray}
The estimates \eqref{(i)}, \eqref{(ii)},  \eqref{(iii)}, \eqref{(iiii)} and \eqref{order2} together with the scaling $\tau=\theta h^\frac{3}{4}$  yield 
\begin{eqnarray}\label{2727}
&&\norm\tilde{b}^{n+1}\norm_{\Omegain}
\le (1+O(\tau)) \norm\tilde{b}^n\norm_\Omegain  \\\nonumber
&&\quad +\frac{O(\tau^2)  \norm\tilde{b}^n\norm_\Omegain +O(\tau h)+O(\tau^2) }{ \norm \tilde{b}^{n+1}\norm_\Omegain} +O(h)\\\nonumber
&&\quad 
-\tau\Big(\sum_{j=1}^3 \norm D_j^+\tilde{b}^{n+1}\norm_{\Omega_h}\Big)\Big(A^{-1}- 
\frac{O(h^\2) \norm\tilde{b}^n\norm_\Omegain
+O(h^\2)+O(\frac{\tau}{h^\2})}{ \norm \tilde{b}^{n+1}\norm_\Omegain}\Big)
\\\nonumber
&&\le(1+\beta_1\tau)\norm\tilde{b}^n\norm_\Omegain 
+\frac{\beta_2\tau h^{\frac{3}{4}}}{\norm \tilde{b}^{n+1}\norm_\Omegain}+\beta_3\tau h^\frac{1}{4}\\\nonumber
&&\quad -\tau\Big(\sum_{j=1}^3 \norm D_j^+\tilde{b}^{n+1}\norm_{\Omega_h}\Big)\Big(A^{-1}-\frac{\beta_4h^\frac{1}{4}}{\norm \tilde{b}^{n+1} \norm_\Omegain}  \Big),
\end{eqnarray}
where $\beta_1$ to $\beta_4$ are some positive constants independent of $\tau ,h$ and $\theta$.

We show by induction, 
\begin{eqnarray}\label{hodai}
&&\norm\tilde{b}^n\norm_\Omegain \le  \eta (1+2\beta_1\tau)^nh^\frac{1}{4}\mbox{\quad for all $0\le n\le T_\tau$ as $\delta=(\tau,h)\to0$,}\\\nonumber
&&\eta:=\max\Big\{\frac{2\beta_3}{\beta_1},A\beta_4\Big\}.
\end{eqnarray}
Since $v^0\in C^1(\bar{\Omega})$ and  
$\norm \tilde{u}^0-v^0\norm_{\Omegain}\le O(h)$, we have \eqref{hodai} for $n=0$. Suppose that \eqref{hodai} holds up to some $n\ge0$ and fails for $n+1$. Then, \eqref{2727} implies for sufficiently small $\delta$,  
\begin{eqnarray*}
&&\norm\tilde{b}^{n+1}\norm_\Omegain\le
 (1+\beta_1\tau)\eta (1+2\beta_1\tau)^nh^\frac{1}{4}
 +\frac{\beta_2\tau h^\frac{3}{4}}{\eta (1+2\beta_1\tau)^{n+1}h^\frac{1}{4}} +\beta_3\tau h^\frac{1}{4}\\
 &&\qquad -\tau\Big(\sum_{j=1}^3 \norm D_j^+\tilde{b}^{n+1}\norm_{\Omega_h}\Big)\Big(A^{-1}-\frac{\beta_4h^\frac{1}{4}}{\eta (1+2\beta_1\tau)^{n+1}h^\frac{1}{4}}  \Big)\\
 &&\quad \le \eta (1+2\beta_1\tau)^{n+1}h^\frac{1}{4}-\eta\beta_1(1+2\beta_1\tau)^n\tau h^\frac{1}{4}
 +\frac{\beta_2\tau h^\2}{\eta} +\beta_3\tau h^\frac{1}{4}\\
 &&\qquad  -\tau\Big(\sum_{j=1}^3 \norm D_j^+\tilde{b}^{n+1}\norm_{\Omega_h}\Big)\Big(A^{-1}-\frac{\beta_4}{\eta }  \Big)\\
 &&\quad \le  \eta (1+2\beta_1\tau)^{n+1}h^\frac{1}{4}
 -\eta \beta_1\tau h^\frac{1}{4}+2\beta_3\tau h^\frac{1}{4}  \le  \eta (1+2\beta_1\tau)^{n+1}h^\frac{1}{4}, 
\end{eqnarray*} 
which is a contradiction. Hence, we necessarily have \eqref{hodai}. 

Note that $(1+2\beta_1\tau)^n\le e^{2\beta_1n\tau}\le e^{2\beta_1T}$ for $0\le n\le T_\tau$ and  $\norm \tilde{b}^n\norm_{\Omega_h}\le \norm \tilde{b}^n\norm_{\Omegain}+O(h^\frac{3}{2})$ because of $\tilde{b}^n=-\bar{v}^n=O(h)$ on $\partial\Omega_h$. Thus, Theorem \ref{error-estimate} follows from \eqref{hodai}. 
\setcounter{section}{3}
\setcounter{equation}{0}
\section{Problem with time-periodic external force}

We investigate the Navier-Stokes equations with a time-periodic external force. Suppose that the external force $f$ is time-periodic with the period $1$, i.e., 
$$f\in L^2_{\rm loc}([0,\infty);L^2(\Omega)^3),\quad f(t,\cdot)=f(1+t,\cdot)\mbox{\quad a.e. $t\ge0$}.$$
Take $\tau=1/T_1$ with $T_1\in\N$. Then, we may introduce the time-$1$ map 
$$\Phi_\delta:\tilde{u}^0\mapsto \tilde{u}^{T_1},\quad\delta=(h,\tau)$$
 of  the discrete Navier-Stokes equations. We find a fixed point of $\Phi_\delta$, which yields a time-periodic solution of the discrete Navier-Stokes equations, i.e., a solution $\bar{u}^n,\tilde{\bar{u}}^n$ of (\ref{initial-1})-(\ref{fractional-3}) such that 
 $$\bar{u}^{n+T_1}=\bar{u}^n,\quad\tilde{\bar{u}}^{n+T_1}=\tilde{\bar{u}}^n\mbox{ for all $n\ge0$}.$$  
Then, we show that  $\bar{u}^n,\tilde{\bar{u}}^n$ tend to a time-periodic Leray-Hopf weak solution with the period $1$ as $\delta\to0$,
where  $v=(v_1,v_2,v_3):[0,\infty)\times\Omega\to\R^3$ is called a time-periodic Leray-Hopf weak solution of
\begin{eqnarray}\label{NSp}
 \left \{
\begin{array}{lll}
\,\,\,\,\, v_t&=& -(v\cdot \nabla)v +\Delta v+f -\nabla p\mbox{\quad in $(0,\infty)\times\Omega $,}
\medskip\\
\nabla\cdot v &=&0\mbox{\quad\quad\quad\quad\quad\quad\quad\quad\quad\quad\quad\,\,\, in $(0,\infty)\times\Omega$, \qquad\qquad\quad}
\medskip\\
\,\,\,\,\,v&=&0\mbox{\qquad\qquad\qquad\qquad\qquad\,\,\,\,\,\,\,\,\, on $\partial \Omega$},
\end{array}
\right.
\end{eqnarray}
with the period $1$, if
\begin{eqnarray*}\nonumber
&&v\in L^\infty([0,\infty);L^2_{\sigma}(\Omega))\cap L^2_{\rm loc}([0,\infty);H^1_{0,\sigma}(\Omega)),\\\nonumber
&& v(t+1,\cdot)=v(t,\cdot)\mbox{ a.e. $t\ge0$},\\
&&- \int_0^\infty\int_\Omega v(t,x)\cdot \partial_t\phi(x,t)dxdt =-\sum_{j=1}^3\int_0^\infty\int_\Omega v_j(t,x)\partial_{x_j}v(t,x)\cdot\phi(t,x)dxdt\\\nonumber
&&\quad -\sum_{j=1}^3\int_0^\infty\int_\Omega\partial_{x_j}v(t,x)\cdot\partial_{x_j}\phi(t,x)dxdt\\\nonumber
&&\quad +\int_0^\infty\int_\Omega f(t,x)\cdot\phi(t,x)dxdt\quad \mbox{ for all $\phi\in C^\infty_0((0,\infty);C^\infty_{0,\sigma}(\Omega))$.}
\end{eqnarray*} 
We also discuss long-time behaviors of  the (discrete) Navier-Stokes equations, as well as an error estimate, assuming that there exists a smooth time-periodic solution of (\ref{NSp}).
\subsection{Time-global solvability and time-periodic solution}
Define the set $\tilde{U}_R$ of initial data $\tilde{u}^0$ of the discrete Navier-Stokes equations as 
$$\tilde{U}_R:=\{\tilde{u}:\Omega_h\to\R^3 \,|\,\,\,\,\norm \tu\norm_{\Omega_h}\le R, \quad \tilde{u}|_{\partial\Omega_h}=0  \},$$ 
and the constant $R_0(\Omega,f)\ge0$ as 
\begin{eqnarray*}
R_0(\Omega,f):=\frac{1}{1-e^{-A^{-2}}}\Big(\frac{1-e^{-2A^{-2}}}{2A^{-2}}\Big)^\2\norm f\norm_{L^2([0,1];L^2(\Omega)^3)},
\end{eqnarray*}
where $A>0$ is the constant in the discrete Poincar\'e type inequality I. The next theorem is an immediate consequence of Theorem \ref{nonperiodic-time-global}.
\begin{Thm}\label{time-global}
For any $R\ge R_0(\Omega,f)$ and fixed $\delta=(h,\tau)$, the time-$1$ map $\Phi_\delta$ maps   
 $\tilde{U}_R$ to itself.  For each $\tu^0\in \tilde{U}_R$, the discrete Navier-Stokes equations is solvable for all $n\in\N$  and the solution satisfies  $\Phi_\delta^m(\tu^{0})=\tu^{mT_1}\in \tilde{U}_R$ for all  $m\in\N$. 
\end{Thm}
\begin{Thm}\label{fixed}
For any $R\ge R_0(\Omega,f)$ and fixed $\delta=(h,\tau)$, the time-$1$ map $\Phi_\delta:\tilde{U}_R\to\tilde{U}_R$ possesses at least one fixed point, which yields a time-periodic solution of the discrete Navier-Stokes equations.   
\end{Thm}
\begin{proof}
Since $N_h:=\sharp\Omega_h$ is finite, we find a one to one onto mapping $\Theta_h:\tilde{U}_R\to B_R\subset\R^{3N_h}$, i.e., we tag the points of $\Omega_h$ as $x^1,x^2,\ldots,x^{N_h}$ and define $y=y(u):=(u(x^1),u(x^2),\ldots,u(x^{N_h}))\in \R^{3N_h}$ for each $u\in\tilde{U}_R$.   Since $\norm u\norm_{\Omega_h}\le R$, the Euclidian norm of $y(u)$ is also bounded by $Rh^{-\frac{3}{2}}$ for each fixed $h$. It is clear that $\tilde{U}_R$ is a convex set, and hence, $B_R$ is a bounded convex subset of $\R^{3N_h}$. 

Since $\Phi_\delta(u)$ is obtained though finitely many basic arithmetic operations, $\Phi_\delta$ is continuous with respect  to $\norm \cdot \norm_{\Omega_h}$. In fact,  let $\tilde{u}^n,u^n$ and $\tilde{w}^n,w^n$ be solutions of the discrete Navier-Stokes equations with $\tilde{u}^0,\tilde{w}^0\in\tilde{U}_R$ and set $\tilde{b}^n:=\tilde{w}^n-\tilde{u}^n$, $b^n:=w^n-u^n$;  It is enough to check that $\norm \tilde{b}^n\norm_{\Omega_h}\to0$ ($n=0,1,2,\ldots,T_1$), as $\norm \tilde{b}^0\norm_{\Omega_h}\to0$ in the sense of $\tilde{w}^0\to\tilde{u}^0$; Since $\norm b^n\norm_{\Omega_h}\le \norm \tilde{b}^n\norm_{\Omega_h}$ due to the property of $P_h$, we have $\norm b^0\norm_{\Omega_h}\to 0$ as $\norm \tilde{b}^0\norm_{\Omega_h}\to0$; Suppose that $\norm b^n\norm_{\Omega_h}\to 0$ as $\norm \tilde{b}^0\norm_{\Omega_h}\to0$ for some $n\ge0$; The discrete Navier-Stokes equations implies \eqref{41414141} in the proof of Theorem \ref{small2} below;  Taking the inner product of \eqref{41414141} and $\tilde{b}^{n+1}$ together with the calculation for (i) and (iii) in the proof, we obtain 
\begin{eqnarray*}
\norm \tilde{b}^{n+1}\norm_{\Omega_h}&\le& \norm b^n\norm_{\Omega_h}+\norm {\rm (ii)}\norm_{\Omega_h}\\
&\le& \norm b^n\norm_{\Omega_h}+\tau\sum_{j=1}^3
\max_{x\in\Omega_h}|D_j\tilde{u}^{n+1}(x)|\norm b^n\norm_{\Omega_h}\to 0\mbox{ as  $\norm \tilde{b}^0\norm_{\Omega_h}\to0$,}
\end{eqnarray*}
where we note that $h$ is fixed and $\max_{x\in\Omega_h}|D_j\tilde{u}^{n+1}(x)|$ is bounded in the process of $\norm \tilde{b}^0\norm_{\Omega_h}\to0$ and $\norm b^{n+1}\norm_{\Omega_h}\le\norm \tilde{b}^{n+1}\norm_{\Omega_h}$; By induction, we have our assertion. 

Therefore, the map 
$\Theta_h\circ\Phi_\delta\circ\Theta_h^{-1}:B_R\to B_R$ 
is continuous with respect to the   Euclidian norm of $\R^{3N_h}$. Brouwer's fixed point theorem guarantees existence of a fixed point.    
\end{proof}
\subsection{Time-periodic Leray-Hopf weak solution}

Let $\tilde{\bar{u}}^n$, $\bar{u}^n$ be the solution of the discrete Navier-Stokes equations with initial data equal to a fixed point of $\Phi_\delta$. Define the step functions $\bar{u}_\delta,\tilde{\bar{u}}_\delta,\bar{w}^i_\delta:[0,\infty)\times\Omega\to\R^3$ with $\tilde{\bar{u}}^n, \bar{u}^n$ and $D^+_i\tilde{\bar{u}}^{n+1}$ in the same way as (\ref{step1}) to (\ref{step3}). The argument on weak and strong convergence in \cite{Kuroki-Soga} proves that  $\bar{u}_\delta|_{[0,1]},\tilde{\bar{u}}_\delta|_{[0,1]}$ weakly converge to some function $\bar{v}\in L^\infty([0,1];L^2_\sigma(\Omega))\cap L^2([0,1];H^1_{0,\sigma}(\Omega))$ and $\tilde{\bar{u}}_\delta|_{[0,1]}$ strongly convergence to $\bar{v}$ in $L^2([0,1];L^2(\Omega)^3)$ as $\delta\to0$ (up to a subsequence). Let $\bar{v}$ be periodically extended in time with the period $1$, i.e., $\bar{v}\in L^\infty([0,\infty);L^2_\sigma(\Omega))\cap L^2_{\rm loc}([0,\infty);H^1_{0,\sigma}(\Omega))$ and $\bar{v}(t+1,\cdot)=\bar{v}(t,\cdot)$ for a.e. $t\in[0,\infty)$. Since  $\tilde{\bar{u}}^n, \bar{u}^n$ are time-periodic with the period $1$, it is clear that, for any fixed $T>0$, $\bar{u}_\delta|_{[0,T]},\tilde{\bar{u}}_\delta|_{[0,T]}$ weakly converge to $\bar{v}|_{[0,T]}$ and $\tilde{\bar{u}}_\delta|_{[0,T]}$ strongly convergence to $\bar{v}|_{[0,T]}$ in $L^2([0,T];L^2(\Omega)^3)$ as $\delta\to0$ (the same subsequence as the above). Furthermore, $\bar{v}$ is a time-periodic Leray-Hopf weak solution of (\ref{NSp}) with the period $1$. By taking $R=R_0(\Omega,f)$ in Theorem \ref{fixed}, we find a time-periodic solution which tends to $0$ as $f\to0$. To sum up, we have 
\begin{Thm}\label{periodicLH}
A time-periodic solution of the discrete Navier-Stokes equations tends to a time-periodic Leray-Hopf weak solution of (\ref{NSp})  as $\delta\to0$ (up to a subsequence).   There exists a family of time-periodic (discrete and Leray-Hopf weak) solutions which tends to $0$ in the $L^2$-norm as $\norm f\norm_{L^2([0,1];L^2(\Omega)^3)}\to0$. 
\end{Thm}
\subsection{Stability of small solution}

We prove that a ``small'' solution is exponentially stable. 
Suppose that there exists  a small solution in the sense of $L^\infty$, i.e.,  a solution $\tilde{u}^n$, $u^n$ of the discrete Navier-Stokes equations \eqref{initial-1}-\eqref{fractional-3}  with $\tilde{u}^0\in\tilde{U}_R$,  $R\ge R_0(\Omega,f)$ such that  
\begin{eqnarray}\label{small}
|\tilde{u}^n(x)|\le\beta_0:=\frac{A^{-1}}{4}\mbox{\quad for all $x\in\Omega_h$ and $n\in\N$},
\end{eqnarray}
where $A$ is the constant from the Poincar\'e type inequality I. We remark that it is not clear when one can find such a solution (nevertheless, one could check with a computer).  
Let $\tilde{w}^n,w^n$ be an arbitrary solution of the discrete Navier-Stokes equations with $\tilde{w}^0\in\tilde{U}_R$ ($\tilde{w}^0\neq\tilde{u}^0$). 
\begin{Thm}\label{small2}
We have 
$$ \norm\tilde{w}^{n}-\tilde{u}^{n}\norm_{\Omega_h}\le e^{-\frac{A^{-2}}{2}n\tau}\norm \tilde{w}^0-\tilde{u}^0\norm_{\Omega_h}\mbox{ for all $n\in\N\cup\{0\}$.}$$
\end{Thm}
\begin{proof}
Set $\tilde{b}^n:=\tilde{w}^n-\tilde{u}^n$, $b^n:=w^n-u^n$. 
Observe that for $x\in\Omegain$,
\begin{eqnarray}\label{41414141}
&&\tilde{b}^{n+1}(x)
=b^n(x) \\\nonumber
&&\quad -\underline{ \frac{\tau}{2}\sum_{j=1}^3   
\Big(w^n_j(x-he^j)D_j\tilde{b}^{n+1}(x-he^j)+w^n_j(x+he^j)D_j\tilde{b}^{n+1}(x+he^j)\Big)}_{\rm (i)} \\\nonumber
&&\quad-\underline{\frac{\tau}{2}\sum_{j=1}^3 \Big(b^n_j(x-he^j)D_j\tilde{u}^{n+1}(x-he^j)+b^n_j(x+he^j)D_j\tilde{u}^{n+1}(x+he^j) \Big)}_{\rm (ii)}\\\nonumber
&&\quad+\underline{\tau\sum_{j=1}^3D^2_j\tilde{b}^{n+1}(x)}_{\rm (iii)}.
\end{eqnarray}
Since $w^n$ is discrete-divergence-free, we have 
\begin{eqnarray*}  
({\rm (i)},\tilde{b}^{n+1})_{\Omega_h}=0.
\end{eqnarray*}
Summation by part yields 
\begin{eqnarray*}
-\sum_{x\in\Omega_h\setminus\partial\Omega_h}({\rm ii})\cdot\tilde{b}^{n+1}(x)h^3&=&-\frac{\tau}{2}\sum_{j=1}^3\sum_{x\in\Omegain} \Big(b^n_j(x-he^j)D_j\tilde{u}^{n+1}(x-he^j)\\
&& +b^n_j(x+he^j)D_j\tilde{u}^{n+1}(x+he^j) \Big)\cdot\tilde{b}^{n+1}(x)h^3\\
&=& \underline{\frac{\tau }{4h}\sum_{j=1}^3\sum_{x\in\Omegain}b^n_j(x)\tilde{u}^{n+1}(x-he^j)\cdot\tilde{b}^{n+1}(x+he^j)h^3} \\
&&\underline{-\frac{\tau }{4h} \sum_{j=1}^3\sum_{x\in\Omegain}b^n_j(x)\tilde{u}^{n+1}(x+he^j)\cdot\tilde{b}^{n+1}(x-he^j)h^3}_{\rm (a)} \\
&& +   \frac{\tau}{2}\sum_{x\in\Omegain} \Big(\sum_{j=1}^3D_jb^n_j(x)\Big)\tilde{u}^{n+1}(x)\cdot\tilde{b}^{n+1}(x) h^3,
\end{eqnarray*}
 \begin{eqnarray*}
{\rm (a)}  &= &\frac{\tau}{4h}\sum_{j=1}^3\sum_{x\in\Omegain}
\Big( 
b^n_j(x)\tilde{u}^{n+1}(x-he^j)\cdot\tilde{b}^{n+1}(x+he^j) \\
&&  -b^n_j(x)\tilde{u}^{n+1}(x-he^j)\cdot\tilde{b}^{n+1}(x-he^j)\\
&& +b^n_j(x)\tilde{u}^{n+1}(x-he^j)\cdot\tilde{b}^{n+1}(x-he^j)
-b^n_j(x)\tilde{u}^{n+1}(x+he^j)\cdot\tilde{b}^{n+1}(x-he^j)
\Big)h^3\\
&=&\frac{\tau}{2}\sum_{j=1}^3\sum_{x\in\Omegain}
b^n_j(x)\tilde{u}^{n+1}(x-he^j)\cdot D_j\tilde{b}^{n+1}(x)h^3\\
&&+ \frac{\tau}{4h}\sum_{j=1}^3\sum_{x\in\Omegain}
\Big( 
b^n_j(x+he^j)\tilde{u}^{n+1}(x)\cdot\tilde{b}^{n+1}(x)\\
&&\qquad\qquad\qquad\qquad\qquad\qquad  -b^n_j(x-he^j)\tilde{u}^{n+1}(x)\cdot\tilde{b}^{n+1}(x-2he^j)
\Big)h^3\\
&=&\frac{\tau}{2}\sum_{j=1}^3\sum_{x\in\Omegain}
b^n_j(x)\tilde{u}^{n+1}(x-he^j)\cdot D_j\tilde{b}^{n+1}(x)h^3\\
&&+ \frac{\tau}{4h}\sum_{j=1}^3\sum_{x\in\Omegain}
\Big( 
b^n_j(x+he^j)\tilde{u}^{n+1}(x)\cdot\tilde{b}^{n+1}(x)\\
&&-b^n_j(x+he^j)\tilde{u}^{n+1}(x)\cdot\tilde{b}^{n+1}(x-2he^j)+b^n_j(x+he^j)\tilde{u}^{n+1}(x)\cdot\tilde{b}^{n+1}(x-2he^j)\\
&&  -b^n_j(x-he^j)\tilde{u}^{n+1}(x)\cdot\tilde{b}^{n+1}(x-2he^j)
\Big)h^3\\
&=&\frac{\tau}{2}\sum_{j=1}^3\sum_{x\in\Omegain}
b^n_j(x)\tilde{u}^{n+1}(x-he^j)\cdot D_j\tilde{b}^{n+1}(x)h^3\\
&&+ \frac{\tau}{2}\sum_{j=1}^3\sum_{x\in\Omegain}
b^n_j(x+he^j)\tilde{u}^{n+1}(x)\cdot D_j\tilde{b}^{n+1}(x-he^j)h^3\\
&&+\frac{\tau}{2}\sum_{x\in\Omegain}\Big(\sum_{j=1}^3 D_jb^n_j(x)\Big)\tilde{u}^{n+1}(x)\cdot\tilde{b}^{n+1}(x-2he^j)h^3. 
 \end{eqnarray*}
Since $b^n$ is discrete-divergence-free, we obtain with \eqref{small}, 
\begin{eqnarray*}
-({\rm (ii)},\tilde{b}^{n+1})_{\Omega_h}
&\le& \tau\beta_0\norm b^n\norm_{\Omega_h}\sum_{j=1}^3\norm D_j\tilde{b}^{n+1}\norm_{\Omega_h}\\
&\le&  \tau\beta_0\norm \tilde{b}^n\norm_{\Omega_h}\sum_{j=1}^3\norm D^+_j\tilde{b}^{n+1}\norm_{\Omega_h} .
\end{eqnarray*}
The Poincar\'e type inequality I implies 
\begin{eqnarray*}
({\rm (iii)},\tilde{b}^{n+1})_{\Omega_h}=-\tau\sum_{j=1}^3\norm D^+_j\tilde{b}^{n+1}\norm^2_{\Omega_h}\le-\tau A^{-1}\sum_{j=1}^3\norm D^+_j\tilde{b}^{n+1}\norm_{\Omega_h}\norm \tilde{b}^{n+1}\norm_{\Omega_h}.
\end{eqnarray*}
Hence, with $\norm b^n\norm_{\Omega_h}\le\norm \tilde{b}^n\norm_{\Omega_h}$, we obtain 
\begin{eqnarray*}
\norm \tilde{b}^{n+1}\norm_{\Omega_h}&\le& \norm \tilde{b}^n\norm_{\Omega_h}
+ \tau\beta_0\frac{\norm \tilde{b}^n\norm_{\Omega_h}}{\norm \tilde{b}^{n+1}\norm_{\Omega_h}}\sum_{j=1}^3\norm D_j\tilde{b}^{n+1}\norm_{\Omega_h}-\tau A^{-1}\sum_{j=1}^3\norm D^+_j\tilde{b}^{n+1}\norm_{\Omega_h}\\
&=& \norm \tilde{b}^n\norm_{\Omega_h}
- \tau\sum_{j=1}^3\norm D^+_j\tilde{b}^{n+1}\norm_{\Omega_h}\Big( A^{-1}-\beta_0\frac{\norm \tilde{b}^n\norm_{\Omega_h}}{\norm \tilde{b}^{n+1}\norm_{\Omega_h}}\Big).
\end{eqnarray*}
Suppose that for $n\ge0$,
$$\frac{ \norm \tilde{b}^{n+1}\norm_{\Omega_h}}{\norm \tilde{b}^{n}\norm_{\Omega_h}}>\frac{1}{1+\tau A^{-2}}.$$
Then, for any sufficiently small $\tau>0$, we have with the Poincar\'e type inequality I,
\begin{eqnarray*}
\norm \tilde{b}^{n+1}\norm_{\Omega_h}&\le&  \norm \tilde{b}^n\norm_{\Omega_h}
- \tau\sum_{j=1}^3\norm D^+_j\tilde{b}^{n+1}\norm_{\Omega_h}\{ A^{-1}-\beta_0(1+\tau A^{-2})    \}\\
&\le& \norm \tilde{b}^n\norm_{\Omega_h}
- \tau\sum_{j=1}^3\norm D^+_j\tilde{b}^{n+1}\norm_{\Omega_h}\frac{A^{-1}}{2}   \\
&\le& \norm \tilde{b}^n\norm_{\Omega_h}
-  \frac{3 }{2}\tau A^{-2}\norm \tilde{b}^{n+1}\norm_{\Omega_h},  
\end{eqnarray*}
which leads to 
\begin{eqnarray*}
\frac{ \norm \tilde{b}^{n+1}\norm_{\Omega_h}}{\norm \tilde{b}^{n}\norm_{\Omega_h}}\le \frac{1}{1+\frac{3}{2}\tau A^{-2}}.
\end{eqnarray*}
This is a contradiction. Therefore, we obtain 
$$ \norm \tilde{b}^{n+1}\norm_{\Omega_h}\le  \frac{1}{1+\tau A^{-2}} \norm \tilde{b}^{n}\norm_{\Omega_h}\le e^{-\frac{A^{-2}}{2}(n+1)\tau}\norm \tilde{b}^{0}\norm_{\Omega_h}\mbox{ for $n=0,1,\cdots$.}$$
\end{proof}
\begin{Cor}\label{nasu}
Suppose that the discrete Navier-Stokes equations possess a solution  satisfying \eqref{small}. Then, a time-periodic solution found in Theorem \ref{fixed} is necessarily unique and bounded by $\beta_0$ for all $x\in\Omega_h$ and $n\in\N$. Furthermore, any other solutions of  the discrete Navier-Stokes equations tend to the time-periodic solution as time goes to infinity.  
\end{Cor}
\begin{proof}
In the proof of Theorem \ref{small2}, take a time-periodic solution as $\tilde{w}^n,w^n$. If there are two different time-periodic solutions, both of them must get arbitrarily close to $\tilde{u}^n,u^n$ satisfying \eqref{small} as $n\to\infty$, which is impossible. The other claims follow from Theorem \ref{small2}.
\end{proof}
Suppose that there exist initial data $v^0\in L^2_\sigma(\Omega)$ and a sequence $\delta\to0$ such that the solution $\tilde{u}^n$ of the discrete Navier-Stokes equations solved with $v^0$ satisfies \eqref{small} for each element $\delta$ of the sequence. 
Then, Theorem \ref{periodicLH} and Corollary \ref{nasu} imply that there exists  a time-periodic Leray-Hopf weak solution that is bounded by $\beta_0$ in the $L^\infty$-sense. 
Furthermore,  since the decay estimate  given in Theorem \ref{small2} is independent of the size of $\delta=(\tau,h)$, we have a stability result on a ``small'' Leray-Hopf weak solution: 
Let $\{\tilde{u}_\delta\}$ be a convergent sequence  of the step functions derived from the above $\tilde{u}^n$, where  $\{\tilde{u}_\delta\}$  tends to a time-global Leray-Hopf weak solution $v$ with initial data $v^0$ as $\delta\to0$. Let $w$ be any time-global Leray-Hopf weak solution with initial data $w^0\in L^2_\sigma(\Omega)$ that can be a limit of  a sequence $\{\tilde{w}_\delta\}$ derived from the discrete Navier-Stokes equations solved with $w^0$, where $\delta\to0$ is the same sequence as the  above. Note that we do not suppose any regularity of $v$ except for the $L^\infty$-bound coming from \eqref{small}. 
\begin{Thm}
We have 
$$\norm w-v\norm_{L^2([t,\infty);L^2(\Omega)^3)}\le \frac{2}{A^{-2}}e^{-\frac{A^{-2}}{2}t}\norm w^0-v^0\norm_{L^2(\Omega)^3}\mbox{\quad for any $t>0$}.$$
\end{Thm}
 \begin{proof} 
Fix an arbitrary $t>0$. Take arbitrary small $\ep>0$ and large $T>t$. Since $\tilde{u}_\delta$ (resp., $\tilde{w}_\delta$) strongly converges to $v$ (resp., $w$) in $\norm \cdot \norm_{L^2([t,T];L^2(\Omega)^3)}$ as $\delta\to0$, we have with sufficiently small $\delta$ and Theorem \ref{small2},
\begin{eqnarray*}
&&\norm w-v\norm_{L^2([t,T];L^2(\Omega)^3)}\le 
\norm w-\tilde{w}_\delta\norm_{L^2([t,T];L^2(\Omega)^3)}+\norm \tilde{w}_\delta-\tilde{u}_\delta\norm_{L^2([t,T];L^2(\Omega)^3)}\\
&&\qquad +\norm \tilde{u}_\delta-v\norm_{L^2([t,T];L^2(\Omega)^3)}\\
&&\quad \le \ep+\norm \tilde{w}_\delta(0,\cdot)-\tilde{u}_\delta(0,\cdot)\norm_{L^2(\Omega)^3}\sum_{n_t\le n<n_T} e^{-\frac{A^{-2}}{2}(n+1)\tau}\tau,
\end{eqnarray*}
where $n_t,n_T\in\N$ are such that $\tau n_t\le t<\tau n_t+\tau$, $\tau n_T\le T<\tau n_T+\tau$.    Taking smaller $\delta$ if necessary, we have 
\begin{eqnarray*}
\norm w-v\norm_{L^2([t,T];L^2(\Omega)^3)}&\le& 2\ep + \norm w^0-v^0\norm_{L^2(\Omega)^3}\int_t^T e^{-\frac{A^{-2}}{2}s}ds \\
&\le& 2\ep  +  \norm w^0-v^0\norm_{L^2(\Omega)^3}   \frac{2}{A^{-2}}e^{-\frac{A^{-2}}{2}t}. 
\end{eqnarray*}
Since $\ep$ and $T$ are arbitrary, we obtain our assertion.
 \end{proof}
\subsection{Error estimate for time-periodic solution in $C^3$-class}

Suppose that there exists a time-periodic (period $1$) solution $\bar{v}$ of the exact Navier-Stokes equations that belongs to the $C^3$-class and satisfies the smallness condition 
\begin{eqnarray}\label{small3}
\max_{x\in\Omega,\,\,t\in[0,1]}|\bar{v}(t,x)|\le \beta_0:=\frac{A^{-1}}{4},
\end{eqnarray}
where $A$ is the constant from the Poincar\'e type inequality I. 
We take $\delta=(\tau,h)$ which satisfies
\begin{eqnarray}\label{scale444}
\tau=\theta h^{\frac{3}{4}},\quad 0<\theta_0\le\theta\le\theta_1,
\end{eqnarray}
where  $\theta_1>\theta_0>0$ is some constant specified later. Let $\tilde{u}^n,u^n$ be any time-global solution of the discrete Navier-Stokes equations. 
\begin{Thm}\label{periodic-error}
There exist constants $\theta_1,\theta_0>0$ and $\beta>0$ for which we have with each $\delta$ sufficiently small,
\begin{eqnarray*}
\norm  \tilde{u}^n-\bar{v}(n\tau,\cdot)\norm_{\Omega_h}\le e^{-\frac{A^{-2}}{2}n\tau}\norm \tilde{u}^0-\bar{v}(0,\cdot)\norm_{\Omega_h}+\beta h^{\frac{1}{4}}
\mbox{\quad for all $n\in\N\cup\{0\}$.}
\end{eqnarray*}
\end{Thm}
\medskip

\noindent{\bf Remark.}  {\it This theorem states that any solution of the discrete Navier-Stokes equations (including time-periodic one!) falls into the $O(h^{\frac{1}{4}})$-neighborhood of the exact time-periodic solution as time goes to infinity. However, it does not claim that a discrete solution tends to a time-periodic state (we do not assume the existence of a discrete solution satisfying \eqref{small}) and hence we do not know about the contraction stated in Theorem  \ref{small2}. }  
\medskip

\begin{proof}
Set  $\bar{v}^n(\cdot):=\bar{v}(n\tau,\cdot)$, $\tilde{b}^n:=\tilde{u}^n-\bar{v}^n$ and $b^n:=u^n-\bar{v}^n$. Observe that we have for $x\in\Omegain$,
\begin{eqnarray}\label{bbbbb2}
&&\tilde{b}^{n+1}(x)
=b^n(x) \\\nonumber
&&\quad -\underline{ \frac{\tau}{2}\sum_{j=1}^3   
\Big(u^n_j(x-he^j)D_j\tilde{b}^{n+1}(x-he^j)+u^n_j(x+he^j)D_j\tilde{b}^{n+1}(x+he^j)\Big)}_{\rm (i)} \\\nonumber 
&&\quad-\underline{\frac{\tau}{2}\sum_{j=1}^3 \Big(b^n_j(x-he^j)D_j\bar{v}^{n+1}(x-he^j)+b^n_j(x+he^j)D_j\bar{v}^{n+1}(x+he^j) \Big)}_{\rm (ii)}\\\nonumber
&&\quad+\underline{\tau\sum_{j=1}^3D^2_j\tilde{b}^{n+1}(x)}_{\rm (iii)} -R^n(x). 
\end{eqnarray}
Following the estimate given in Section 3, we have \eqref{v}, \eqref{(i)}, \eqref{(iii)}, \eqref{(iiii)} and \eqref{order2} (with $\bar{v}$ instead of $v$) also for \eqref{bbbbb2}. We estimate $-({\rm (ii)},\tilde{b}^{n+1})_{\Omegain}$ by taking $\max |\bar{v}^{n+1}|$ out of the inner product after ``summation by parts''. For this purpose, observe that 
\begin{eqnarray*}
&&-\sum_{x\in\Omega_h\setminus\partial\Omega_h}({\rm ii})\cdot\tilde{b}^{n+1}(x)h^3\\
&&\quad =  \frac{\tau}{2}\sum_{x\in\Omegain} \Big(\sum_{j=1}^3D_jb^n_j(x)\Big)\bar{v}^{n+1}(x)\cdot\tilde{b}^{n+1}(x) h^3 \\
&&\qquad +  \underline{\frac{\tau }{4h}\sum_{j=1}^3\sum_{x\in\Omegain}b^n_j(x)\bar{v}^{n+1}(x-he^j)\cdot\tilde{b}^{n+1}(x+he^j)h^3} \\
&&\qquad\underline{-\frac{\tau }{4h} \sum_{j=1}^3\sum_{x\in\Omegain}b^n_j(x)\bar{v}^{n+1}(x+he^j)\cdot\tilde{b}^{n+1}(x-he^j)h^3}_{\rm (a)} \\
&&\qquad +\frac{\tau }{4h}\sum_{j=1}^3\sum_{x\in\Gamma^{j-}_h}b^n_j(x)\bar{v}^{n+1}(x-he^j)\cdot\tilde{b}^{n+1}(x+he^j)h^3\\
&&\qquad -\frac{\tau }{4h}\sum_{j=1}^3\sum_{x\in\Gamma^{j+}_h}b^n_j(x-he^j)\bar{v}^{n+1}(x-2he^j)\cdot\tilde{b}^{n+1}(x)h^3\\
&&\qquad +\frac{\tau }{4h}\sum_{j=1}^3\sum_{x\in\Gamma^{j-}_h}b^n_j(x+he^j)\bar{v}^{n+1}(x+2he^j)\cdot\tilde{b}^{n+1}(x)h^3\\
&&\qquad -\frac{\tau }{4h}\sum_{j=1}^3\sum_{x\in\Gamma^{j+}_h}b^n_j(x)\bar{v}^{n+1}(x+he^j)\cdot\tilde{b}^{n+1}(x-he^j)h^3.
\end{eqnarray*}
 We have  
 \begin{eqnarray*}
{\rm (a)}  &= &\frac{\tau}{4h}\sum_{j=1}^3\sum_{x\in\Omegain}
\Big( 
b^n_j(x)\bar{v}^{n+1}(x-he^j)\cdot\tilde{b}^{n+1}(x+he^j) \\
&&  -b^n_j(x)\bar{v}^{n+1}(x-he^j)\cdot\tilde{b}^{n+1}(x-he^j)\\
&& +b^n_j(x)\bar{v}^{n+1}(x-he^j)\cdot\tilde{b}^{n+1}(x-he^j)
-b^n_j(x)\bar{v}^{n+1}(x+he^j)\cdot\tilde{b}^{n+1}(x-he^j)
\Big)h^3\\
&=&\frac{\tau}{2}\sum_{j=1}^3\sum_{x\in\Omegain}
b^n_j(x)\bar{v}^{n+1}(x-he^j)\cdot D_j\tilde{b}^{n+1}(x)h^3\\
&&+ \frac{\tau}{4h}\sum_{j=1}^3\sum_{x\in\Omegain}
\Big( 
b^n_j(x)\bar{v}^{n+1}(x-he^j)\cdot\tilde{b}^{n+1}(x-he^j)\\
&&\qquad\qquad\qquad\qquad\qquad\qquad  -b^n_j(x)\bar{v}^{n+1}(x+he^j)\cdot\tilde{b}^{n+1}(x-he^j)
\Big)h^3\\
&=&\frac{\tau}{2}\sum_{j=1}^3\sum_{x\in\Omegain}
b^n_j(x)\bar{v}^{n+1}(x-he^j)\cdot D_j\tilde{b}^{n+1}(x)h^3\\
&&+ \frac{\tau}{4h}\sum_{j=1}^3\sum_{x\in\Omegain}
\Big( 
b^n_j(x+he^j)\bar{v}^{n+1}(x)\cdot\tilde{b}^{n+1}(x)\\
&&\qquad\qquad\qquad\qquad\qquad\qquad  -b^n_j(x-he^j)\bar{v}^{n+1}(x)\cdot\tilde{b}^{n+1}(x-2he^j)\Big)h^3\\
&&+\frac{\tau}{4h}\sum_{j=1}^3\sum_{x\in\Gamma^{j-}_h}
b^n_j(x+he^j)\bar{v}^{n+1}(x)\cdot\tilde{b}^{n+1}(x) h^3       \\
&&-\frac{\tau}{4h}\sum_{j=1}^3\sum_{x\in\Gamma^{j+}_h}   
b^n_j(x)\bar{v}^{n+1}(x-he^j)\cdot\tilde{b}^{n+1}(x-he^j) h^3\\
&&+\frac{\tau}{4h}\sum_{j=1}^3\sum_{x\in\Gamma^{j-}_h} b^n_j(x)\bar{v}^{n+1}(x+he^j)\cdot\tilde{b}^{n+1}(x-he^j) h^3     \\
&&-\frac{\tau}{4h}\sum_{j=1}^3\sum_{x\in\Gamma^{j+}_h}   
b^n_j(x-he^j)\bar{v}^{n+1}(x)\cdot\tilde{b}^{n+1}(x-2he^j) h^3   \\
&=&\frac{\tau}{2}\sum_{j=1}^3\sum_{x\in\Omegain}
b^n_j(x)\bar{v}^{n+1}(x-he^j)\cdot D_j\tilde{b}^{n+1}(x)h^3\\
&&+ \frac{\tau}{2}\sum_{j=1}^3\sum_{x\in\Omegain}
b^n_j(x+he^j)\bar{v}^{n+1}(x)\cdot D_j\tilde{b}^{n+1}(x-he^j)h^3\\
&&+\frac{\tau}{2}\sum_{x\in\Omegain}\Big(\sum_{j=1}^3 D_jb^n_j(x)\Big)\bar{v}^{n+1}(x)\cdot\tilde{b}^{n+1}(x-2he^j)h^3\\
&&+\frac{\tau}{4h}\sum_{j=1}^3\sum_{x\in\Gamma^{j-}_h}
b^n_j(x+he^j)\bar{v}^{n+1}(x)\cdot\tilde{b}^{n+1}(x) h^3       \\
&&-\frac{\tau}{4h}\sum_{j=1}^3\sum_{x\in\Gamma^{j+}_h}   
b^n_j(x)\bar{v}^{n+1}(x-he^j)\cdot\tilde{b}^{n+1}(x-he^j) h^3\\
&&+\frac{\tau}{4h}\sum_{j=1}^3\sum_{x\in\Gamma^{j-}_h} b^n_j(x)\bar{v}^{n+1}(x+he^j)\cdot\tilde{b}^{n+1}(x-he^j) h^3     \\
&&-\frac{\tau}{4h}\sum_{j=1}^3\sum_{x\in\Gamma^{j+}_h}   
b^n_j(x-he^j)\bar{v}^{n+1}(x)\cdot\tilde{b}^{n+1}(x-2he^j) h^3 . 
 \end{eqnarray*}
Since $\sum_jD_jb_j^n(x)=\sum_jD_ju_j^n(x)-\sum_jD_j\bar{v}_j^n(x)=-\sum_jD_j\bar{v}_j^n(x)=O(h^2)$, we have 
\begin{eqnarray*}
&&\!\!\!\!\!\!\! \frac{\tau}{2}\sum_{x\in\Omegain} \!\!\!\Big(\sum_{j=1}^3D_jb^n_j(x)\Big)\bar{v}^{n+1}(x)\cdot\tilde{b}^{n+1}(x) h^3 \le O(\tau h^2)\norm\tilde{b}^{n+1} \norm_{\Omegain},\\
&&\!\!\!\!\!\!\! \frac{\tau}{2}\sum_{x\in\Omegain}\!\!\!\Big(\sum_{j=1}^3 D_jb^n_j(x)\Big)\bar{v}^{n+1}(x)\cdot\tilde{b}^{n+1}(x-2he^j)h^3\le O(\tau h^2)\norm\tilde{b}^{n+1} \norm_{\Omegain}+O(\tau h^5), 
\end{eqnarray*}
where $O(\tau h^5)$ comes from the values outside $\Omega_h$.  The terms with $\Gamma^{j\pm}_h$ are estimated as
\begin{eqnarray*}
&&\frac{\tau }{4h}\sum_{j=1}^3\sum_{x\in\Gamma^{j-}_h}b^n_j(x)\bar{v}^{n+1}(x-he^j)\cdot\tilde{b}^{n+1}(x+he^j)h^3\\
&&\qquad -\frac{\tau }{4h}\sum_{j=1}^3\sum_{x\in\Gamma^{j+}_h}b^n_j(x-he^j)\bar{v}^{n+1}(x-2he^j)\cdot\tilde{b}^{n+1}(x)h^3\\
&&=\frac{\tau }{4h}\sum_{j=1}^3\sum_{x\in\Gamma^{j-}_h}\big(0-\bar{v}^n_j(x)\big)\bar{v}^{n+1}(x-he^j)\cdot\big(\tilde{u}^{n+1}(x+he^j)-\bar{v}^{n+1}(x+he^j)\big)h^3\\
&&\qquad -\frac{\tau }{4h}\sum_{j=1}^3\sum_{x\in\Gamma^{j+}_h}\big(u^n_j(x-he^j)-\bar{v}^n_j(x-he^j))\bar{v}^{n+1}(x-2he^j)\cdot\big(0-\bar{v}^{n+1}(x)\big)h^3\\
&&\le \frac{\tau}{4h}O(h)O(h)\sum_{j=1}^3\sum_{x\in\Omega_h}|\tilde{u}^{n+1}(x+he^j)\chi_{\partial\Omega_h}(x)|h^3+\frac{\tau}{4h}O(h)O(h)O(h)\sum_{j=1}^3(\sharp\Gamma^{j-}_h)h^3\\
&&\qquad + \frac{\tau}{4h}\sum_{j=1}^3\sum_{x\in\Omega_h}|u^{n}(x-he^j)\chi_{\partial\Omega_h}(x)|h^3O(h)O(h)+\frac{\tau}{4h}O(h)O(h)O(h)\sum_{j=1}^3(\sharp\Gamma^{j-}_h)h^3\\
&&=O(\tau h)h^\2+O(\tau h^3)=O(\tau h^{\frac{3}{2}}).
\end{eqnarray*}
Therefore,  with \eqref{small3} and 
$\norm b^n\norm_{\Omega_h}\le \norm b^n\norm_{\Omegain}+O(h^\frac{3}{2})\le  \norm \tilde{b}^n\norm_{\Omegain}+O(h)$, we obtain
\begin{eqnarray*}
&&-\sum_{x\in\Omegain}({\rm ii})\cdot \tilde{b}^{n+1}(x)h^3
\le \tau\beta_0\norm b^n\norm_{\Omega_h}\sum_{j=1}^3\norm D^+_j\tilde{b}^{n+1}\norm_{\Omega_h}\\\nonumber
&&\qquad +O(\tau h^2)\norm\tilde{b}^{n+1} \norm_{\Omegain}+O(\tau h^{\frac{3}{2}})\\
&&\quad \le  \tau\beta_0\norm \tilde{b}^n\norm_{\Omegain}\sum_{j=1}^3\norm D^+_j\tilde{b}^{n+1}\norm_{\Omega_h}
+O(\tau h) \sum_{j=1}^3\norm D^+_j\tilde{b}^{n+1}\norm_{\Omega_h}\\\nonumber
&&\qquad +O(\tau h^2)\norm\tilde{b}^{n+1} \norm_{\Omegain}+O(\tau h^{\frac{3}{2}}).
\end{eqnarray*}
This estimate together with \eqref{(i)}, \eqref{(iii)},  \eqref{(iiii)} ($\bar{v}$ instead of $v$) and the scaling condition of $\delta=(\tau,h)$ lead to 
\begin{eqnarray}\label{iv}
&&\norm \tilde{b}^{n+1}\norm_{\Omegain}\le  (1+\beta_1\tau h^\2)
\norm \tilde{b}^n\norm_{\Omegain}+\tau\Big( \frac{\beta_2\theta_1h^\frac{3}{4}}{\norm\tilde{b}^{n+1}\norm_{\Omegain}}+\beta_3\theta_0^{-1} h^{\frac{1}{4}}   \Big)\\
\nonumber
&&\quad-\tau\Big(\sum_{j=1}^3\norm D^+_j\tilde{b}^{n+1}\norm_{\Omega_h}\Big)\Big(
A^{-1}-\beta_0\frac{\norm \tilde{b}^n\norm_{\Omegain}}{\norm \tilde{b}^{n+1}\norm_{\Omegain}}-\frac{\beta_4\theta_1 h^{\frac{1}{4}}}{\norm \tilde{b}^{n+1}\norm_{\Omegain}}
\Big),
\end{eqnarray}
where $\beta_1,\beta_2,\cdots$ are some positive constants independent of $\delta=(\tau,h)$, $n$, $\theta$, $\theta_0$ and $\theta_1$. 
\begin{Lemma}\label{last4}
Suppose that 
$$\theta_1>\sqrt{\frac{A\beta_3}{\beta_4}},\quad \theta_0=\frac{A\beta_3}{\beta_4\theta_1}$$
 in \eqref{scale444} (the first inequality guarantees that $\theta_0<\theta_1$). Then,  we have for each sufficiently small $\delta$,
\begin{eqnarray}\label{error44}
\norm \tilde{b}^n\norm_{\Omegain}\le e^{-\frac{A^{-2}}{2}n\tau}\norm \tilde{b}^0\norm_{\Omegain}+8A \theta_1\beta_4h^{\frac{1}{4}}
\mbox{\quad for all $n\in\N\cup\{0\}$.}
\end{eqnarray}
\end{Lemma}
\begin{proof}
Set $\eta:=8A \theta_1\beta_4$. For $n=0$, \eqref{error44} holds. Suppose that \eqref{error44} holds up to some $n$ and fails to hold for $n+1$. Then, \eqref{iv} implies 
\begin{eqnarray*}
\norm \tilde{b}^{n+1}\norm_{\Omegain}&\le&  (1+\beta_1\tau h^\2)
\norm \tilde{b}^n\norm_{\Omegain}+\tau \frac{\beta_2\theta_1h^\frac{1}{2}}{\eta}+\tau\beta_3\theta_0^{-1} h^{\frac{1}{4}}\\
&&-\tau\Big(\sum_{j=1}^3\norm D^+_j\tilde{b}^{n+1}\norm_{\Omega_h}\Big)\Big( 
A^{-1}-\frac{3}{2}\beta_0
-\frac{\beta_4\theta_1}{\eta}
\Big),
\end{eqnarray*}
where we have    
$$A^{-1}-\frac{3}{2}\beta_0
-\frac{\beta_4\theta_1}{\eta}\ge \frac{A^{-1}}{2}.$$
Hence, we may apply the Poincar\'e type inequality I with the correction \eqref{Poin-poin} to obtain 
\begin{eqnarray*}
-\tau\Big(\sum_{j=1}^3\norm D^+_j\tilde{b}^{n+1}\norm_{\Omega_h}\Big)\Big( 
A^{-1}-\frac{3}{2}\beta_0
-\frac{\beta_4\theta_1}{\eta}
\Big)\le -\tau \frac{3A^{-2}}{2}\norm \tilde{b}^{n+1}\norm_{\Omegain}+\beta_5\tau h^\2.
\end{eqnarray*}
Therefore, we have for sufficiently small $\delta$,
\begin{eqnarray*}
\norm \tilde{b}^{n+1}\norm_{\Omegain}&\le&  \frac{1+\tau\beta_1h^\2}{1+\tau\frac{3A^{-2}}{2}}
\norm \tilde{b}^n\norm_{\Omegain}+\tau\frac{\beta_2\theta_1h^\2}{\eta} +\tau\beta_2\theta_0^{-1}h^\frac{1}{4}+\beta_5\tau h^\2 \\
&\le& \frac{1}{1+\tau A^{-2}}
\norm \tilde{b}^n\norm_{\Omegain}+2\tau\beta_3\theta_0^{-1} h^{\frac{1}{4}}\\
&\le&e^{-\frac{A^{-2}}{2}\tau}\norm \tilde{b}^n\norm_{\Omegain}+2\tau\beta_3\theta_0^{-1} h^{\frac{1}{4}}\\
&\le& e^{-\frac{A^{-2}}{2}(n+1)\tau}\norm\tilde{b}^0\norm_{\Omegain}+ e^{-\frac{A^{-2}}{2}\tau }\eta h^\frac{1}{4}+ 2\tau\beta_3\theta_0^{-1} h^{\frac{1}{4}} \\
&\le& e^{-\frac{A^{-2}}{2}(n+1)\tau}\norm\tilde{b}^0\norm_{\Omegain}+ \eta h^{\frac{1}{4}}+\Big(2\beta_3\theta_0^{-1}-\frac{A^{-2}}{4}\eta\Big)\tau h^{\frac{1}{4}}\\
&=&e^{-\frac{A^{-2}}{2}(n+1)\tau}\norm\tilde{b}^0\norm_{\Omegain}+ \eta h^{\frac{1}{4}},
\end{eqnarray*}
which is a contradiction.
\end{proof}
Since $\tilde{b}^n=-\bar{v}^n=O(h)$ on $\partial\Omega_h$ and $\norm \tilde{b}^n\norm_{\Omega_h}\le \norm \tilde{b}^n\norm_{\Omegain}+O(h^\frac{3}{2})$, Theorem \ref{periodic-error} is an immediate consequence of Lemma  \ref{last4}. 
\end{proof}
\setcounter{section}{4}
\setcounter{equation}{0}
\section{Problem with periodic boundary conditions}
We briefly summarize results on the problems in $\Omega=\T^3$, i.e., the problems with the periodic boundary conditions. By taking $h=1/N$ with $N\in\N$, one can formulate the discrete Navier-Stokes equations with the periodic boundary conditions in the same way as Section 2. 

Since there is no boundary of $\Omega=\T^3$, the Poincar\'e type inequality II is obtained in a simpler way (see \cite{Chorin}) and we are not bothered by the remaining terms from the boundary in the arguments corresponding to  Section 3. Hence, we may optimize our error estimates by the central difference and the diffusive scaling $\tau=O(h^2)$. 
In fact,  Lemma \ref{div-smooth} is improved to be 
$$\norm u-P_hu\norm_{\Omega_h}\le O( h^4)=O(\tau h^2), \quad u\in C^5_{\sigma}(\T^3).$$      
Furthermore, Theorem \ref{error-estimate} is improved to be 
$$\norm \tilde{u}^{n}-v(\tau n,\cdot)\norm_{\Omega_h}\le O(h^2),$$
provided  an exact solution $v$ belongs to the $C^5$-class. Then, using the inequality 
$$\max_{\Omega_h}|u(x)|h^{\frac{3}{2}} \le \norm u\norm_{\Omega_h},$$
we obtain the $L^\infty$-error estimate 
$$\max_{\Omega_h}| \tilde{u}^{n}(x)-v(\tau n,x)|\le O(\sqrt{h}).$$   
\indent The results in Section 4 are also improved  with the diffusive scaling and with a $C^5$-exact solution, where we need to argue with initial data with a common average over $\T^3$ (the average of a solution is conserved both for the Navier-Stokes equations and discrete Navier-Stokes equations). In particular, Theorem \ref{periodic-error} becomes 
\begin{eqnarray*}
\norm \tilde{u}^n-\bar{v}(n\tau,\cdot)\norm_{\Omega_h}\le e^{-\frac{A^{-2}}{2}n\tau}\norm \tilde{u}^0-\bar{v}(0,\cdot)\norm_{\Omega_h}+\beta h^{2}
\mbox{\quad for all $n\in\N\cup\{0\}$.}
\end{eqnarray*}
Then, we have an $L^\infty$-estimate of $\tilde{b}^n$ to be $O(\sqrt{h})$ for all sufficiently large $n$. This implies that there exists a solution of the discrete Navier-Stokes equations that satisfies \eqref{small}, provided there exists an exact smooth time-periodic solution $\bar{v}$ that satisfies (\ref{small3}).   Hence, we obtain 
\begin{Thm}
Suppose that there exists an exact time-periodic solution $\bar{v}\in C^5([0,\infty)\times\T^3)$ satisfying \eqref{small3}. Then, a time-periodic discrete solution $\bar{u}^n$,  $\tilde{\bar{u}}^n$ with the same average as $\bar{v}$ is unique and asymptotically stable within initial data with the same average. The $L^\infty$-error between $\bar{v}(\tau n,\cdot)$ and $\tilde{\bar{u}}^n$  is  $O(\sqrt{h})$ for all $n$.  
\end{Thm} 
\noindent Therefore, one can approximate a time-periodic discrete solution and exact one only by solving an initial value problem of the discrete Navier-Stokes equations for a long time.  
\medskip\medskip\medskip\medskip\medskip

\noindent{\bf Acknowledgement.} The second author, Kohei Soga,   is supported by JSPS Grant-in-aid for Young Scientists \#18K13443.

\medskip\medskip

\appendix
\def\thesection{Appendix}
\section{}

\begin{proof}[{\rm  1.}  Proof of Lemma \ref{Poincare2}]
It is enough to prove that we have a constant  $\tilde{A}>0$ depending only on $\Omega$ for which 
$$\sum_{x\in\Omega_h^{\circ j}} |\phi(x)-[\phi]^j|^2
\le \tilde{A}^2 \sum_{x\in\Omega_h\setminus\partial\Omega_h}  |\D \phi(x)|^2$$
hold for each $j$.

We first  find such a constant $\tilde{A}=\tilde{A}_h$ with fixed $h$:  Suppose that there is no such constant $A_h$. Then, for each $k\in\N$, we have $\phi_k:\Omega_h\to\R$ such that 
\begin{eqnarray*}
\sum_{x\in\Omega_h^{\circ j}} |\phi_k(x)- [\phi_k]^j|^2 \ge 
k\sum_{x\in\Omega_h\setminus\partial\Omega_h} |\D \phi_k(x)|^2.
\end{eqnarray*}
We normalize $\phi_k$ as 
$$\psi_k(x):=\frac{\phi_k(x)- [\phi_k]^j}{\dis \Big(\sum_{x\in\Omega_h^{\circ j}} |\phi_k(x)- [\phi_k]^j|^2\Big)^\2}  .$$
Then, we see that 
\begin{eqnarray*}
\sum_{x\in\Omega_h^{\circ j}}\psi_k(x)=0,\quad
 \sum_{x\in\Omega_h^{\circ j}}|\psi_k(x)|^2=1,\quad 
 \sum_{x\in\Omega_h\setminus\partial\Omega_h}|\D\psi_k(x)|^2\le k^{-1} 
 \quad\mbox{for all $k$},
\end{eqnarray*}
which implies that $\psi_k$ is bounded on $\Omega_h^{\circ j}$. Furthermore, since $x+he^i$ belongs to $\Omegain$ for any $x\in\Omega_h^{\circ j}$, we have for $i=1,2,3$,
$$k^{-1}\ge\sum_{\Omegain}| \D\psi_k(x)|^2\ge\sum_{x\in\Omega_h^{\circ j}} |D_i\psi_k(x+he^i)|^2\ge\Big|\frac{\psi_k(x+2he^i)-\psi(x)}{2h}\Big| \mbox{ for all $x\in\Omega_h^{\circ j}$},$$ 
which implies that $\psi_k$ is bounded on  
$$B_h:=\Omega_h^{\circ j}\cup \{ x+2he^i\,|\,x\in\Omega_h^{\circ j},\,\,\,\,\,i=1,2,3 \}.$$
Hence, since $h$ is fixed, we have a subsequence of $\{\psi_k\}$ whose restriction on $B_h$ converges to some $w:B_h\to\R$. We have 
$$\sum_{x\in\Omega_h^{\circ j}}w(x)=0,\quad \sum_{x\in\Omega_h^{\circ j}}|w(x)|^2=1,\quad \sum_{i=1}^3\sum_{x\in\Omega_h^{\circ j}}\Big|\frac{w(x+2he^i)-w(x)}{2h}\Big|^2=0.$$  
Since $\Omega_h^{\circ j}$ is connected, this is a contradiction. 

We next prove that there exists $\tilde{A}>0$ such that $\tilde{A}_h\le \tilde{A}$ for $h\to0+$. Suppose that there is no such $\tilde{A}$. 
Then, for each $k\in\N$, we have $h_k>0$ and $\phi_k:\Omega_{h_k}\to\R$ such that $h_k\to 0$ as $k\to\infty$ and 
$$\sum_{x\in\Omega_{h_k}^{\circ j}} |\phi_k(x)- [\phi_k]^j|^2 \ge k\sum_{x\in\Omega_{h_k}\setminus\partial\Omega_{h_k}} |\D \phi_k(x)|^2$$
We normalize $\phi_k$ as 
$$\psi_k(x):=\frac{\phi_k(x)- [\phi_k]^j}{\dis\Big(\sum_{x\in\Omega_{h_k}^{\circ j}} |\phi_k(x)- [\phi_k]^j|^2(2h_k)^3 \Big)^\2 } .$$
Then, we see that 
\begin{eqnarray*}
&&\sum_{x\in\Omega_{h_k}^{\circ j}}\psi_k(x)(2h_k)^3=0,\quad
\sum_{x\in\Omega_{h_k}^{\circ j}}|\psi_k(x)|^2(2h_k)^3=1,\\ 
&&\sum_{x\in\Omega_{h_k}\setminus\partial\Omega_{h_k}}|\D\psi_k(x)|^2(2h_k)^3\le k^{-1}(2h_k)^3 
\quad\mbox{for all $k$}.
\end{eqnarray*}
Set 
\begin{eqnarray*}
&&\Theta_{h_k}:=\bigcup_{x\in \Omega_{h_k}^{\circ j}} [x_1,x_1+2h_k)\times[x_2,x_2+2h_k)\times[x_3,x_3+2h_k).
\end{eqnarray*}
Let $\hat{w}_k:\Theta_{h_k}\to\R^3$ be the step function defined as  
$$
\hat{w}_k(x):=\psi_{k}(y)\mbox{ for $x\in[y_1,y_1+2h_k)\times[y_2,y_2+2h_k)\times[y_3,y_3+2h_k)$, $y\in\Omega_{h_k}^{\circ j}$}.
$$
Let $w_k:\Theta_{h_k}\to\R$ be the Lipschitz interpolation of the step function derived from $\psi^k|_{\Omega_{h_k}^{\circ j}}$ as Lemma \ref{Lip-interpolation2}.  
Then, we have 
\begin{eqnarray}\nonumber
&&\norm w_k-\hat{w}_k\norm_{L^2(\Theta_{h_k})}=O(h_k)k^{-1}(2h_k)^3,\quad \\\nonumber
&&\Big|\int_{\Theta_{h_k}}w_k(x)dx\Big|=\Big|\int_{\Theta_{h_k}}w_k(x)-\hat{w}_kdx\Big|\le \tilde{K}\norm w_k-\hat{w}_k\norm_{L^2(\Theta_{h_k})},\\\label{232323}
&&\norm \partial_{x_i}w_k\norm_{L^2(\Theta_{h_k})}\le Kk^{-1}(2h_k)^3
\quad(i=1,2,3)\quad\mbox{ for all $k$},
\end{eqnarray}
where $K,\tilde{K}>0$ are some constant, which leads to
\begin{eqnarray}\label{2323231}
&&\norm w_k\norm_{L^2(\Theta_{h_k})}= 1+O(h_k)h_kk^{-1}(2h_k)^3,\quad \\\label{2323232}
&&\int_{\Theta_{h_k}}w_k(x)dx=  O(h_k)k^{-1}(2h_k)^3.
\end{eqnarray}
 We extend $w_k$ to be a function $\bar{w}_k$ of $H^1(\Omega)$ with the estimates 
\begin{eqnarray}\label{23232323}
\norm \bar{w}_k \norm_{H^1(\Omega)}\le L\norm w_k\norm_{H^1(\Theta_{h_k})},\quad \norm \partial_x\bar{w}_k \norm_{L^2(\Omega)^3}\le L\norm \partial_x w_k\norm_{L^2(\Theta_{h_k})^3},
\end{eqnarray}
where  $L>0$ is a constant independent from  $k$. This is possible because $\Omega$ is bounded and Lipschitz: Let $\Gamma_1,\ldots, \Gamma_M$ be a family of open balls covering $\partial\Omega$ such that each $\partial\Omega\cap \Gamma_m$ is described as the graph of  a Lipschitz map $y_3=\varphi_m(y_1,y_2)$, where $(y_1,y_2,y_3)$ is a Cartesian coordinate pointing to $\tilde{e}^1, \tilde{e}^2,\tilde{e}^3$ in the original space spanned by $\{e^1,e^2,e^3\}$ in such a way that  $|\tilde{e}^3\cdot e^i|$ is uniformly away from $1$ for $i=1,2,3$, and   all $\varphi_m$ ($m=1,\ldots,M$) has a common Lipschitz constant (the coordinate $(y_1,y_2,y_3)$ depends on $m$); For all $k\in\N$ large enough,  $\Gamma_1,\ldots, \Gamma_M$ cover also  $\partial\Theta_{h_k}$, where  $\partial\Theta_{h_k}$ consists of $2h_k$-squares orthogonal to $e^1$, $e^2$  or $e^3$; 
Each $\partial\Theta_{h_k}\cap \Gamma_m$ is arbitrarily close to $\partial\Omega\cap \Gamma_m$ as $k\to\infty$; 
We see that $\partial\Theta_{h_k}\cap \Gamma_m$ is described as the graph of a Lipschitz map $y_3=\tilde{\varphi}_m(y_1,y_2;k)$; 
We see also  that $\tilde{\varphi}_m(\cdot;k)$ has a common Lipschitz constant for all $m$ and $k$; Then, we may apply the standard extension argument for $H^1$-functions to obtain \eqref{23232323}. 
Since $\bar{w}_k$,  $\partial_{x_i}\bar{w}_k$ are bounded in $L^2(\Omega)$, we have subsequences, still denoted by the same symbol, which weakly converge to some $w\in L^2(\Omega)$, $v_i\in L^2(\Omega)^3$, respectively. For each $f\in C_0^\infty(\Omega)$, we have 
\begin{eqnarray*}
\int_{\Omega}w(x)\partial_{x_i} f(x) dx &=& \int_{\Omega}(w(x)-\bar{w}_k(x))\partial_{x_i} f(x) dx-\int_{\Omega}\partial_{x_i}\bar{w}_k(x) f(x) dx \\
&&\to-\int_{\Omega}v_i(x) f(x) dx \mbox{\quad as $k\to\infty$,} 
\end{eqnarray*}
which implies that  $w\in H^1(\Omega)$ with $\partial_{x_i} w=v_i$.  
On the other hand, the Rellich-Kondrachov theorem yields a subsequence of $\bar{w}_k$, still denoted by the same symbol, such that $\bar{w}_k$ strongly converges to  $w$ in $L^2(\Omega)$ as $k\to\infty$.
 For each $f\in C_0^\infty(\Omega)$, we have with \eqref{232323},
\begin{eqnarray*}
\int_\Omega \partial_{x_i}w(x)f(x)dx&=&\int_\Omega \partial_{x_i}\bar{w}_k(x)f(x)dx +\int_\Omega (\partial_{x_i}w(x)-\partial_{x_i}\bar{w}_k(x))f(x) dx\\
&=& \int_\Omega \partial_{x_i}\bar{w}_k(x)f(x)dx -\int_\Omega (w(x)-\bar{w}_k(x))\partial_{x_i}f(x) dx\\
&&\to0\mbox{\quad as $k\to\infty$.}
\end{eqnarray*}
Hence, we obtain $\partial_x w=0 $ a.e. in $\Omega$, which implies that $w$ is constant in $\Omega$. 
Since $\norm \bar{w}_k\norm_{L^2(\Omega)}\ge \norm w_k\norm_{L^2(\Theta_{k})}=1+O(h_k)k^{-1}h_k^3$ due to \eqref{2323231}, we see that $\norm w\norm_{L^2(\Omega)}\ge1$ and $w\equiv a \neq 0$ in $\Omega$. 
This is a contradiction, since \eqref{2323232} implies 
\begin{eqnarray*}
\int_\Omega w(x)dx
&=&\int_{\Theta_{h_k}} w(x)dx+a{\rm meas}[\Omega\setminus\Theta_{h_k}]\\
&=&\int_{\Theta_{h_k}} w_k(x)dx+\int_{\Theta_{h_k}} (w(x)-w_k(x))dx +a{\rm meas}[\Omega\setminus\Theta_{h_k}]\\
&&\to0 \mbox{\quad as $k\to\infty$.}
\end{eqnarray*}
\end{proof}
\medskip\medskip
\medskip

\begin{proof}[{\rm 2. } Proof of Lemma \ref{Lip-interpolation2} ] 
 For each $y=(y_1,y_2,y_3)\in\Omega_h^{\circ j}$, define the following functions:
\begin{eqnarray*}
&&f_1(x_1):[y_1,y_1+2h]\to\R,\quad \\
&&f_1(x_1):=u(y)+\frac{u(y+2he^1)-u(y)}{2h}(x_1-y_1);\\
&&f_2(x_1):[y_1,y_1+2h]\to\R,\quad \\
&&f_2(x_1):=u(y+2he^2)+\frac{u(y+2he^2+2he_1)-u(y+2he^2)}{2h}(x_1-y_1);\\
&&f_3(x_1,x_2):[y_1,y_1+2h]\times[y_2,y_2+2h]\to\R, \quad \\
&&f_3(x_1,x_2):= f_1(x_1)+\frac{f_2(x_1)-f_1(x_1)}{2h}(x_2-y_2);\\
&&g_1(x_1):[y_1,y_1+2h]\to\R,\quad \\
&&g_1(x_1):=u(y+2he^3)+\frac{u(y+2he^3+2he^1)-u(y+2he^3)}{2h}(x_1-y_1);\\
&&g_2(x_1):[y_1,y_1+2h]\to\R,\quad \\
&&g_2(x_1):=u(y+2he^3+2he^2)\\
&&\qquad\qquad+\frac{u(y+2he^3+2he^2+2he_1)-u(y+2he^3+2he^2)}{2h}(x_1-y_1);\\
&&g_3(x_1,x_2):[y_1,y_1+2h]\times[y_2,y_2+2h]\to\R, \quad \\
&&g_3(x_1,x_2):= g_1(x_1)+\frac{g_2(x_1)-g_1(x_1)}{2h}(x_2-y_2);\\ 
&&w(x_1,x_2,x_3):C_{2h}^+(y)\to\R,\\
&&w(x_1,x_2,x_3):=f_3(x_1,x_2)+\frac{g_3(x_1,x_2)-f_3(x_1,x_2)}{2h}(x_3-y_3)
\end{eqnarray*}
Then, we see that 
\begin{eqnarray*}
&&\!\!\!\!\!w(x_1,x_2,x_3)= u(y)+D_1u(y+he^1)(x_1-y_1)+ D_2u(y+he^2)(x_2-y_2)\\
&&\qquad+D_3u(y+he^3)(x_3-y_3)\\
&&\qquad+\{ D_1u(y+2he^2+he^1) -  D_1u(y+he^1) \}\frac{(x_1-y_1)(x_2-y_2)}{2h}\\
&&\qquad +\{ D_1u(y+2he^3+he^1)-D_1u(y+he^1)\} \frac{(x_1-y_1)(x_3-y_3)}{2h}\\
&&\qquad + \{D_2u(y+2he^3+he^2)-D_2u(y+he^2)\}\frac{(x_2-y_2)(x_3-y_3)}{2h}\\
&&\qquad +\{ D_1u(y+2he^2+2he^3+he^1)-D_1u(y+2he^3+he^1) \\
&&\qquad -D_1u(y+2he^2+he^1)+D_1u(y+he^1)\}\frac{(x_1-y_1)(x_2-y_2)(x_3-y_3)}{(2h)^2}. 
%
\end{eqnarray*}
It is clear that $w$ can be Lipschitz continuously connected with each other, yielding $w:\Theta_h^j\to\R$ that satisfies the inequalities.  
\end{proof}


\end{document}